\documentclass{article}

\usepackage{amsthm,amsmath,natbib}
\usepackage{hyperref}

\usepackage{npbayes}

\allowdisplaybreaks

\title{Posteriors, conjugacy, and exponential families for completely random measures}
\author{Tamara Broderick \and Ashia C.~Wilson \and Michael I.~Jordan}

\begin{document}

\maketitle

\begin{abstract}
We demonstrate how to calculate posteriors for general
Bayesian nonparametric priors and likelihoods based on completely
random measures (CRMs). We further show how to represent Bayesian
nonparametric priors as a sequence of finite draws using a
size-biasing approach---and how to represent full Bayesian
nonparametric models via finite marginals. Motivated by conjugate
priors based on exponential family representations of likelihoods, we
introduce a notion of exponential families for CRMs, which we call
exponential CRMs. This construction allows us to specify automatic
Bayesian nonparametric conjugate priors for exponential CRM
likelihoods. We demonstrate that our exponential CRMs allow
particularly straightforward recipes for size-biased and marginal
representations of Bayesian nonparametric models. Along the way, we
prove that the gamma process is a conjugate prior for the Poisson
likelihood process and the beta prime process is a conjugate prior for
a process we call the odds Bernoulli process. We deliver a size-biased
representation of the gamma process and a marginal representation of
the gamma process coupled with a Poisson likelihood process.
\end{abstract}

\section{Introduction}

An important milestone in Bayesian analysis was the development of a general
strategy for obtaining conjugate priors based on exponential family representations
of likelihoods~\citep{degroot:1970:optimal}.  While slavish adherence to exponential-family conjugacy
can be criticized, conjugacy continues to occupy an important place in Bayesian analysis,
for its computational tractability in high-dimensional problems and for its role in inspiring
investigations into broader classes of priors (e.g., via mixtures, limits, or augmentations).
The exponential family is, however, a parametric class of models, and it is of interest
to consider whether similar general notions of conjugacy can be developed for Bayesian
nonparametric models.  Indeed, the nonparametric literature is replete with nomenclature
that suggests the exponential family, including familiar names such as ``Dirichlet,''
``beta,'' ``gamma,'' and ``Poisson.'' These names refer to aspects of the random measures
underlying Bayesian nonparametrics, either the L\'evy measure used in constructing certain
classes of random measures or properties of marginals obtained from random measures.
In some cases, conjugacy results have been established that parallel results from classical 
exponential families;
in particular, the Dirichlet process is known to be conjugate to
a multinomial process likelihood~\citep{ferguson:1973:bayesian}, the beta process is
conjugate to a Bernoulli process~\citep{kim:1999:nonparametric, thibaux:2007:hierarchical}
and to a negative binomial process~\citep{broderick:2015:combinatorial}.  Moreover,
various useful representations for marginal distributions, including stick-breaking
and size-biased representations, have been obtained by making use of properties that
derive from exponential families.  It is striking, however, that these results have
been obtained separately, and with significant effort; a general formalism that encompasses
these individual results has not
yet emerged.
In this paper, we provide the single, holistic framework so strongly 
suggested by the nomenclature.  Within this single framework, we show that it is 
straightforward to calculate posteriors and establish conjugacy.  Our framework includes 
the specification of a Bayesian nonparametric analog of the finite exponential family, 
which allows us to provide automatic and constructive nonparametric conjugate priors 
given a likelihood specification as well as general recipes for marginal and size-biased 
representations.

A broad class of Bayesian nonparametric priors---including those built on the Dirichlet 
process~\citep{ferguson:1973:bayesian}, the beta process~\citep{hjort:1990:nonparametric}, 
the gamma process~\citep{ferguson:1973:bayesian,
lo:1982:bayesian,
titsias:2008:infinite},
and the negative binomial 
process~\citep{zhou:2012:beta,broderick:2015:combinatorial}---can be viewed as
models for the allocation of data points to traits.  These processes give us pairs 
of traits together with rates or frequencies with which the traits occur in some 
population.  Corresponding likelihoods assign each data point in the population 
to some finite subset of traits conditioned on the trait frequencies.  What makes 
these models nonparametric is that the number of traits in the prior is countably 
infinite. Then the (typically random) number of traits to which any individual 
data point is allocated is unbounded, but also there are always new traits to 
which as-yet-unseen data points may be allocated. That is, such a model allows 
the number of traits in any data set to grow with the size of that data set.

A principal challenge of working with such models arises in posterior inference.
There is a countable infinity of trait frequencies in the prior which we must 
integrate over to calculate the posterior of trait frequencies given allocations 
of data points to traits.  Bayesian nonparametric models sidestep the full 
infinite-dimensional integration in three principal ways: conjugacy, size-biased 
representations, and marginalization.

In its most general form, conjugacy simply asserts that the prior is in the same 
family of distributions as the posterior.  When the prior and likelihood are in 
finite-dimensional conjugate exponential families, conjugacy can turn posterior 
calculation into, effectively, vector addition.  As a simple example, consider 
a model with beta-distributed prior, $\prw \sim \tb(\prw | \alpha, \beta)$,
for some fixed hyperparameters $\alpha$ and $\beta$. For the likelihood, let 
each observation $\obsw_{n}$ with $n \in \{1,\ldots,N\}$ be iid Bernoulli-distributed
conditional on parameter $\prw$: $\obsw_{n} \iid \bern(\obsw | \prw)$.  Then the 
posterior is simply another beta distribution, $\tb(\prw | \alpha_{post}, \beta_{post})$, 
with parameters updated via addition: $\alpha_{post} := \alpha + \sum_{n=1}^{N} x_{n}$ 
and $\beta_{post} := \beta + N - \sum_{n=1}^{N} x_{n}$.  While conjugacy is certainly 
useful and popular in the case of finite parameter cardinality, there is arguably a 
stronger computational imperative for its use in the infinite-parameter case.
Indeed, the core prior-likelihood pairs of Bayesian nonparametrics are generally
proven \citep{hjort:1990:nonparametric, kim:1999:nonparametric,lo:1982:bayesian,
thibaux:2007:hierarchical, broderick:2015:combinatorial},
or assumed to be \citep{titsias:2008:infinite, thibaux:2008:nonparametric},
conjugate. When such proofs exist, though, thus far they have been specialized 
to specific pairs of processes.  In what follows, we demonstrate a general way 
to calculate posteriors for a class of distributions that includes all of these
classical Bayesian nonparametric models.  We also define a notion of exponential 
family representation for the infinite-dimensional case and show that, given a 
Bayesian nonparametric exponential family likelihood, we can readily construct 
a Bayesian nonparametric conjugate prior.

Size-biased sampling provides a finite-dimensional distribution for each of the 
individual prior trait frequencies \citep{thibaux:2007:hierarchical,paisley:2010:stick}.
Such a representation has played an important role in Bayesian nonparametrics in
recent years, allowing for either exact inference via slice
sampling \citep{damien:1999:gibbs,neal:2003:slice}---as demonstrated
by \citet{teh:2007:stick,broderick:2015:combinatorial}---or approximate 
inference via truncation \citep{doshi:2009:variational, paisley:2011:variational}.
This representation is particularly useful for building hierarchical models 
\citep{thibaux:2007:hierarchical}. 
We show that our framework yields such representations in general, and we show
that our construction is especially straightforward to use in the exponential 
family framework that we develop.

Marginal processes avoid directly representing the infinite-dimensional prior and 
posterior altogether by integrating out the trait frequencies. Since the trait 
allocations are finite for each data point, the marginal processes are finite 
for any finite set of data points. Again, thus far, such processes have been shown 
to exist separately in special cases; for example, the Indian buffet process 
\citep{griffiths:2006:infinite} is the marginal process for the beta process prior 
paired with a Bernoulli process likelihood \citep{thibaux:2007:hierarchical}. 
We show that the integration that generates the marginal process from the full 
Bayesian model can be generally applied in Bayesian nonparametrics and takes a 
particularly straightforward form when using conjugate exponential family priors 
and likelihoods. We further demonstrate that, in this case, a basic, constructive 
recipe exists for the general marginal process in terms of only finite-dimensional 
distributions.

Our results are built on the general class of stochastic processes known as
\emph{completely random measures} (CRMs) \citep{kingman:1967:completely}.
We review CRMs in \mysec{crm} and we discuss what assumptions are needed 
to form a full Bayesian nonparametric model from CRMs in \mysec{bnp}.
Given a general Bayesian nonparametric prior and likelihood (\mysec{prior_like}), 
we demonstrate in \mysec{post} how to calculate the posterior.  Although 
the development up to this point is more general, we next introduce a concept 
of exponential families for CRMs (\mysec{exp_fam_crm}) and call such models 
\emph{exponential CRMs}. We show that we can generate automatic conjugate priors given 
exponential CRM likelihoods in \mysec{auto_conj}. Finally, we show how we can
generate recipes for size-biased representations (\mysec{size}) and marginal 
processes (\mysec{marg}), which are particularly straightforward in the exponential 
CRM case (\cor{size} in \mysec{size} and \cor{marg} in \mysec{marg}).
We illustrate our results on a number of examples and derive new conjugacy 
results, size-biased representations, and marginal processes along the way.

We note that some similar results have been obtained by \citet{orbanz:2010:conjugate}
and \citet{james:2014:poisson}. In the present work, we focus on creating 
representations that allow tractable inference.

\section{Bayesian models based on completely random measures} \label{sec:gen}

As we have discussed, we view Bayesian nonparametric models as being composed
of two parts: (1) a collection of pairs of traits together with their frequencies 
or rates and (2) for each data point, an allocation to different traits.
Both parts can be expressed as \emph{random measures}.  Recall that a random 
measure is a random element whose values are measures.

We represent each trait by a point $\prl$ in some space $\locsp$ of traits.
Further, let $\prw_{k}$ be the frequency, or rate, of the trait represented by $\prl_{k}$,
where $k$ indexes the countably many traits. In particular, 
$\prw_{k} \in \obssp$. Then $(\prw_{k}, \prl_{k})$ is a tuple
consisting of the frequency of the $k$th trait together with its trait descriptor.
We can represent the full collection
of pairs of traits with their frequencies by the discrete measure on $\locsp$ that places weight
$\prw_{k}$ at location $\prl_{k}$:
\begin{equation} \label{eq:prm}
	\prm = \sum_{k=1}^{K} \prw_{k} \delta_{\prl_{k}},
\end{equation}
where the cardinality $K$ may be finite or infinity.

Next, we form data point $\obsm_{n}$ for the $n$th individual.
The data point $\obsm_{n}$ is viewed as a discrete measure. 
Each atom of $\obsm_{n}$ represents a pair consisting of 
(1) a trait to which the $n$th individual is allocated and 
(2) a degree to which the $n$th individual is allocated to this 
particular trait. That is,
\begin{equation} \label{eq:obsm}
	\obsm_{n} = \sum_{k=1}^{K_{n}} \obsw_{n,k} \delta_{\prl_{n,k}},
\end{equation}
where again $\prl_{n,k} \in \locsp$ represents a trait and now $\obsw_{n,k} \in \obssp$
represents the degree to which the $n$th data point
belongs to trait $\prl_{n,k}$. $K_{n}$ is the total number of traits to which the $n$th data point belongs.

Here and in what follows, we treat $\obsm_{1:N} = \{\obsm_{n}: n \in [N]\}$ as our observed data points
for $[N] := \{1,2,3,\ldots,N\}$.
In practice
$\obsm_{1:N}$ is often
incorporated into a more complex Bayesian hierarchical model. For instance,
in topic modeling, $\prl_{k}$ represents a topic; that is, $\prl_{k}$ is a 
distribution over words in a vocabulary~\citep{blei:2003:latent,teh:2006:hierarchical}.
$\prw_{k}$ might 
represent the frequency with which the topic $\prl_{k}$ occurs in a corpus 
of documents. $\obsw_{n,k}$ might be a positive integer and represent the number of words in topic
$\prl_{n,k}$ that occur in the $n$th document. So the $n$th document has a total length of
$\sum_{k=1}^{K_{n}} \obsw_{n,k}$ words. In this case, the actual observation consists of the words in
each document, and the topics are latent. Not only are the results concerning posteriors, conjugacy,
and exponential family representations that we develop below useful for inference in such models, but
in fact our results are especially useful in such models---where
the traits and any ordering on the traits are not known in advance.

Next, we want to specify a full Bayesian model for our data points $\obsm_{1:N}$. To do so,
we must first define a prior distribution for the random measure $\prm$ as well as a 
likelihood for each random measure $\obsm_{n}$ conditioned on $\prm$.
We let $\siglocsp$ be a $\sigma$-algebra of subsets of $\locsp$, where we assume all singletons
are in $\siglocsp$. Then we consider random measures $\prm$ and $\obsm_{n}$ whose values
are measures on $\locsp$. Note that for any random measure $\prm$ and any
measurable set $A \in \siglocsp$, $\prm(A)$ is a random variable.

\subsection{Completely random measures} \label{sec:crm}

We can see from \eqs{prm} and \eqss{obsm} that we desire a distribution on 
random measures that yields discrete measures almost surely. A particularly 
simple form of random measure called a \emph{completely random measure}
can be used to generate
a.s.\ discrete random measures~\citep{kingman:1967:completely}.

A completely random measure $\prm$ is defined as a random measure that satisfies
one additional
property; for any disjoint, measurable sets $A_{1}, A_{2}, \ldots, A_{K} \in \siglocsp$, we require that
$\prm(A_{1}), \prm(A_{2}), \ldots, \prm(A_{K})$ be independent random variables.
\citet{kingman:1967:completely}
showed that a completely random measure can always be decomposed into a sum of
three independent parts:
\begin{equation} \label{eq:crm_parts}
	\prm = \prm_{det} + \prm_{fix} + \prm_{ord}.
\end{equation}
Here, $\prm_{det}$ is the deterministic component, $\prm_{fix}$ is the \emph{fixed-location}
component, and $\prm_{ord}$ is the \emph{ordinary} component.
In particular, $\prm_{det}$ is any deterministic measure. We define the remaining two parts next.

The fixed-location
component is called the ``fixed component'' by \citet{kingman:1967:completely},
but we expand the name slightly here to
emphasize that $\prm_{fix}$ is defined to be constructed from a set of random weights at 
fixed (i.e., deterministic) locations. That is,
\begin{equation} \label{eq:fix_discrete}
	\prm_{fix} = \sum_{k=1}^{K_{fix}} \prw_{fix, k} \delta_{\prl_{fix, k}},
\end{equation}
where the number of
fixed-location atoms, $K_{fix}$, may be either finite or infinity;
$\prl_{fix,k}$ is deterministic, and $\prw_{fix,k}$
is a non-negative, real-valued random variable (since $\Phi$
is a measure).
Without loss of generality,
we assume that the locations $\prl_{fix, k}$ are all distinct. Then,
by the independence assumption of CRMs, we must have that $\prw_{fix, k}$
are independent random variables across $k$.
Although the fixed-location atoms are often ignored
in the Bayesian nonparametrics literature,
we will see that the fixed-location component has a key role
to play in establishing Bayesian nonparametric conjugacy
and in the CRM representations we present.

The third and final component is the ordinary component. 
Let $\#(A)$ denote the cardinality of some countable set $A$.
Let $\fullratem$ be any $\sigma$-finite, deterministic measure on
$\obssp \times \locsp$, where $\obssp$ is equipped
with the Borel $\sigma$-algebra and $\Sigma_{\obssp \times \locsp}$
is the resulting product $\sigma$-algebra given $\Sigma_{\locsp}$.
Recall that a \emph{Poisson
point process} with rate measure $\fullratem$ on $\obssp \times \locsp$
is a random countable subset $\Pi$ of $\obssp \times \locsp$ such that two
properties hold \citep{kingman:1992:poisson}:
\begin{enumerate}
	\item For any $A \in \Sigma_{\obssp \times \locsp}$,
		$\#(\Pi \cap A) \sim \pois(\fullratem(A))$.
	\item For any disjoint $A_1, A_2, \ldots, A_K \in \Sigma_{\obssp \times \locsp}$,
		$\#(\Pi \cap A_1), \#(\Pi \cap A_2), \cdots, \#(\Pi \cap A_K)$ are
independent random variables.
\end{enumerate}
To generate an ordinary component, start
with a Poisson point process on $\obssp \times \locsp$, characterized
by its rate measure $\fullratem(d\prw \times d\prl)$. This process yields $\Pi$,
a random and countable
set of points: $\Pi = \{(\prw_{ord,k}, \prl_{ord,k}) \}_{k=1}^{K_{ord}}$, where $K_{ord}$ may be
finite or infinity.
Form the ordinary component measure by letting $\prw_{ord,k}$ be the weight of the atom located at
$\prl_{ord,k}$:
\begin{equation} \label{eq:ord_discrete}
	\prm_{ord} = \sum_{k=1}^{K_{ord}} \prw_{ord,k} \delta_{\prl_{ord,k}}.
\end{equation}

Recall that we stated at the start of \mysec{crm} that CRMs may be used to
produce a.s.\ discrete
random measures. To check this assertion, note that $\prm_{fix}$ is a.s.\ 
discrete by construction (\eq{fix_discrete}) and $\prm_{ord}$ is a.s.\ discrete
by construction (\eq{ord_discrete}). 
$\prm_{det}$ is the one component that may not be a.s.\ atomic.
Thus the prevailing norm in using models based on CRMs is to set $\prm_{det} \equiv 0$;
in what follows, we adopt
this norm.
If the reader is concerned about missing any atoms in $\prm_{det}$, note that it is straightforward to adapt 
the treatment of $\prm_{fix}$ to include the case where the atom weights are deterministic.
When we set
$\prm_{det} \equiv 0$,
we are left with 
$\prm = \prm_{fix} + \prm_{ord}$ by \eq{crm_parts}.
So $\prm$ is also discrete, as desired.

\subsection{Prior and likelihood} \label{sec:prior_like}

The prior that we place on $\prm$ will be a fully general CRM (minus any deterministic 
component) with one additional assumption on the rate measure of the ordinary component.
Before incorporating the additional assumption, we say that $\prm$ has 
a fixed-location component with $K_{fix}$ atoms, where the
$k$th atom has arbitrary distribution $\dfixwk$: $\prw_{fix,k} \indep \dfixwk(d\prw)$.
$K_{fix}$ may be finite or infinity, and $\prm$ has an ordinary component characterized 
by rate measure $\fullratem(d\prw \times d\prl)$.  The additional assumption we make 
is that the distribution on the weights in the ordinary component is assumed to be
decoupled from the distribution on the locations. That is, the rate measure decomposes
as
\begin{equation} \label{eq:rate_factorize}
	\fullratem(d\prw \times d\prl) = \wratem(d\prw) \cdot \dordloc(d\prl),
\end{equation}
where $\wratem$ is any $\sigma$-finite, deterministic measure on 
$\obssp$ and $\dordloc$ is any proper distribution on $\locsp$.
While the distribution over locations has been discussed 
extensively elsewhere \citep{neal:2000:markov,wang:2013:variational}, it
is the weights that affect the allocation of data points to traits. 

Given the factorization of $\fullratem$ in \eq{rate_factorize},
 the ordinary component of $\prm$ can be generated by letting 
$\{\prw_{fix,k}\}_{k=1}^{K_{ord}}$ be the points of a Poisson point 
process generated on $\mathbb{R}_{+}$ with rate 
$\wratem$.\footnote{Recall that $K_{ord}$ may be finite or infinity
depending on $\wratem$ and is random when taking finite values.}
We then draw the locations $\{\prl_{fix,k}\}_{k=1}^{K_{ord}}$
iid according to $\dordloc(d\prl)$: $\prl_{fix,k} \iid \dordloc(d\prl)$.
Finally, for each $k$, $\prw_{fix,k} \delta_{\prl_{fix,k}}$ is an atom in 
$\prm_{ord}$.
This factorization will allow us to focus our attention on the trait frequencies,
and not the trait locations,
in what follows. Moreover, going forward, we will assume $\dordloc$ is diffuse
(i.e., $\dordloc$ has no atoms)
so that the ordinary component atoms are all at a.s.\ distinct locations, which are further
a.s.\ distinct from the fixed locations.

Since we have seen that $\prm$ is an a.s.\ discrete random measure, we can write
it as
\begin{equation} \label{eq:discrete_prior}
	\prm = \sum_{k=1}^{K} \prw_{k} \delta_{\prl_{k}},
\end{equation}
where $K := K_{fix} + K_{ord}$
may be finite or infinity, and every $\prl_{k}$ is a.s.\ unique.
That is, we will sometimes find it helpful notationally to use \eq{discrete_prior}
instead of separating the fixed and ordinary components.
At this point, we have specified the prior for $\prm$ in our general model.

Next, we specify the likelihood; i.e., we specify how to generate the data points
$\obsm_{n}$ given $\prm$.
We will
assume each $\obsm_{n}$ is generated iid given $\prm$ across the data indices $n$.
We will let $\obsm_{n}$ be a CRM with only a fixed-location component
given $\prm$. In particular, 
the atoms of $\obsm_{n}$ will be located at the atom locations of $\prm$,
which are fixed when we condition on $\prm$:
$$
	\obsm_{n} := \sum_{k=1}^{K} \obsw_{n,k} \delta_{\prl_{k}}.
$$
Here, $\obsw_{n,k}$ is drawn according to some distribution $\dlike$ that may
take $\prw_{k}$, the weight of $\prm$ at location $\prl_{k}$,
as a parameter; i.e.,
\begin{equation} \label{eq:alloc_distr}
	\obsw_{n,k} \indep \dlike(d\obsw | \prw_{k}) \quad \textrm{independently across $n$ and $k$}.
\end{equation}

Note that while every atom of $\obsm_{n}$ is located at an atom of $\prm$,
it is not necessarily the case that every atom of $\prm$ has a corresponding atom
in $\obsm_{n}$. In particular, if $\obsw_{n,k}$ is zero for any $k$, there is no atom
in $\obsm_{n}$ at $\prl_{k}$.

We highlight that the model above stands in contrast to Bayesian nonparametric
\emph{partition} models, for which there is a large literature. In partition models (or clustering models), $\prm$ is
a random probability measure \citep{ferguson:1974:prior};
in this case,
the probability constraint precludes $\prm$ from being a completely
random measure, but it is often chosen to be a normalized completely random
measure \citep{james:2009:posterior, lijoi:2010:models}. The choice of Dirichlet process (a normalized gamma process) for $\prm$ is particularly popular due
to a number of useful properties that coincide in this single choice \citep{doksum:1974:tailfree, escobar:1994:estimating, escobar:1995:bayesian, escobar:1998:computing, ferguson:1973:bayesian, lo:1984:class, maceachern:1994:estimating, perman:1992:size, pitman:1996:random, pitman:1996:some, sethuraman:1994:constructive, west:1994:hierarchical}.
In partition models, $\obsm_{n}$ is a draw from the probability distribution described by $\prm$.
If we think of such $\obsm_n$ as a random measure, it is a.s.\ a single unit mass at a point $\prl$ with
strictly positive
probability in $\prm$.

One potential connection between these two types of models is provided by combinatorial clustering
\citep{broderick:2015:combinatorial}. In partition models, we might suppose that we have a number
of data sets, all of which we would like to partition. For instance, in a document modeling scenario, 
each document might be a data set; in particular each data point is a word in the document. And we might wish to 
partition the words in each document. An alternative perspective is to suppose that there is a single data set,
where each data point is a document. Then the document exhibits traits with multiplicities, where the multiplicities
might be the number of words from each trait; typically a trait in this application would be a topic. In this case,
there are a number of other names besides feature or trait model that may be applied to the overarching model---such
as admixture model or mixed membership model \citep{airoldi:2014:handbook}.

\subsection{Bayesian nonparametrics} \label{sec:bnp}

So far we have described a prior and likelihood that may be used to form a Bayesian
model.
We have already stated above that forming a \emph{Bayesian nonparametric}
model
imposes some restrictions on the prior and likelihood. We 
formalize these restrictions
in \asus{fix}, \asuss{inf}, and \asuss{fin} below.

Recall that the premise of Bayesian
nonparametrics is that the number of traits represented in a collection of data
can grow with the number of data points.
More explicitly, we achieve the desideratum that the number of traits is unbounded,
and may always grow as new data points are collected,
by modeling a countable infinity of traits. This assumption requires
that the prior have a countable infinity of atoms. These must either be fixed-location atoms
or ordinary component atoms. 
Fixed-location atoms represent known traits in some sense since we must know
the fixed locations of the atoms in advance. Conversely, ordinary component atoms
represent unknown traits, as yet to be discovered, since both their locations and associated
rates are unknown a priori.
Since we cannot know (or represent) a countable infinity of
traits a priori, we cannot start with a countable infinity of fixed-location atoms. 
\begin{enumerate}[label=A\arabic{*}., ref=A\arabic{*},leftmargin=5.0em,start=0]
	\item \label{as:fix} The number of fixed-location atoms in $\prm$ is finite.
\end{enumerate}
Since we require a countable infinity of traits in total and they cannot come from the
fixed-location atoms by \asu{fix},
the ordinary component must contain a countable infinity of atoms. This assumption
will be true if and only if the rate measure on the trait frequencies has infinite mass.
\begin{enumerate}[label=A\arabic{*}., ref=A\arabic{*},leftmargin=5.0em,start=1]
	\item \label{as:inf} $\wratem(\obssp) = \infty$.
\end{enumerate}

Finally, an implicit part of the starting premise is
that each data point be allocated to only a finite number of traits;
we do not expect to glean an infinite amount of information from finitely represented
data. Thus, we require that the number of atoms in every $\obsm_{n}$ be finite.
By \asu{fix}, the number of atoms in $\obsm_{n}$ that correspond to fixed-location atoms
in $\prm$ is finite. But by \asu{inf}, the number of atoms in $\prm$ from the ordinary component
is infinite. So there must be some restriction on the distribution of values of 
$\obsm$ at the atoms of $\prm$ (that is, some 
restriction on $\dlike$ in \eq{alloc_distr}) such that only finitely many
of these values are nonzero. 

In particular, note that if $\dlike(d\obsw | \prw)$ does not contain an atom at zero
for any $\prw$, then 
a.s.\ every one of the countable infinity of atoms of $\obsm$ will be nonzero.
Conversely, it follows that, for our desiderata to hold, we must
have that $\dlike(d\obsw | \prw)$ exhibits an atom at zero.
One consequence of this observation is that
$\dlike(d\obsw | \prw)$ cannot be purely continuous
for all $\prw$.
Though this line of reasoning
does not necessarily preclude a mixed continuous and discrete $\dlike$,
we henceforth assume that $\dlike(d\obsw | \prw)$ is discrete, with support
$\Zstar = \{0,1,2,\ldots\}$, for all $\prw$. 

In what follows, we write 
$\denslike(\obsw | \prw)$ for the probability mass function of $\obsw$ given $\prw$.
So our requirement  
that each data point be allocated to only a finite number of traits
translates into a requirement that the number of atoms of $\obsm_{n}$ with
values in $\Zplus = \{1,2,\ldots\}$ be finite. Note that, by construction,
the pairs $\{(\prw_{ord,k}, \obsw_{ord,k})\}_{k=1}^{K_{ord}}$ form a marked Poisson point process
with rate measure
$\fullratem_{mark}(d\prw \times d\obsw) := \wratem(d\prw) \denslike(\obsw | \prw)$.
And the pairs with
$\obsw_{ord,k}$ equal to any particular value $\obsw \in \Zplus$ further form a thinned
Poisson point process with rate measure
$\wratem_{\obsw}(d\prw) := \wratem(d\prw) \denslike(\obsw | \prw)$.
In particular, the number of atoms of $\obsm$ with weight $\obsw$
is Poisson-distributed with mean $\wratem_{\obsw}(\obssp)$.
So the number of atoms of $\obsm$ is finite if and only if
the following assumption holds.\footnote{When we have the more general
case of a mixed continuous and discrete $\dlike$,
\asu{fin} becomes
\begin{enumerate}[label=A2b., ref=A2b, leftmargin=5.0em]
	\item \label{as:fin_full} $\int_{\obsw > 0} \int_{\prw \in \mathbb{R}_{+}}
		\wratem(d\prw) \dlike(d\obsw | \prw) < \infty$.
\end{enumerate}
}
\begin{enumerate}[label=A\arabic{*}., ref=A\arabic{*},leftmargin=5.0em,start=2]
	\item \label{as:fin} $\sum_{\obsw=1}^{\infty} \wratem_{\obsw}(\obssp) < \infty$
		for $\wratem_{\obsw} := \wratem(d\prw) \denslike(\obsw | \prw)$.
\end{enumerate}

Thus \asus{fix}, \asuss{inf}, and \asuss{fin} capture our Bayesian nonparametric
desiderata. We illustrate the development so far with an example.

\begin{example} \label{ex:bp_bep}
	The \emph{beta process} \citep{hjort:1990:nonparametric}
	provides an example distribution for $\prm$. In its most general form,
	sometimes called the \emph{three-parameter beta process}
	\citep{teh:2009:indian, broderick:2012:beta},
	the beta process has an ordinary component
	whose weight rate measure has a beta distribution kernel,
	\begin{equation} \label{eq:bp_ord}
		\wratem(d\prw) = \massp
			\prw^{-\discp-1} (1-\prw)^{\concp + \discp - 1} d\prw,
	\end{equation}
	with support on $(0,1]$.
	Here, the three fixed hyperparameters are $\massp$, the \emph{mass parameter}; 
	$\concp$, the \emph{concentration parameter}; and $\discp$,
	the \emph{discount parameter}.\footnote{
	In \citep{teh:2009:indian, broderick:2012:beta}, the ordinary component
	features the beta distribution kernel in \eq{bp_ord} multiplied not only by $\massp$
	but also by
	a more complex, positive, real-valued expression in $\concp$ and $\discp$. Since
	all of $\massp$, $\concp$, and $\discp$ are fixed hyperparameters,
	and $\massp$ is an arbitrary positive real value, any other constant
	factors containing the
	hyperparameters can be absorbed into $\massp$, as in the main text here.}
	Moreover, each of its $K_{fix}$ fixed-location atoms, $\prw_{k} \delta_{\prl_{k}}$, has
	a beta-distributed weight \citep{broderick:2015:combinatorial}:
	\begin{equation} \label{eq:bp_fix}
		\prw_{fix,k} \sim \tb(\prw | \rho_{fix,k}, \sigma_{fix,k}),
	\end{equation}
	where $\rho_{fix,k}, \sigma_{fix,k} > 0$ are fixed hyperparameters of the model.
	
	By \asu{fix}, $K_{fix}$ is finite. By \asu{inf},
	$\wratem(\obssp) = \infty$. To achieve this infinite-mass restriction,
	the beta kernel in \eq{bp_ord}
	must be improper; i.e., either $-\discp \le 0$ or $\concp + \discp \le 0$. Also, note that
	we must have $\massp > 0$
	since $\wratem$ is a measure (and the case $\massp = 0$ would be trivial).
	
	Often the beta process is used as 
	a prior paired with a \emph{Bernoulli process} likelihood
	\citep{thibaux:2007:hierarchical}. The Bernoulli
	process specifies that, given $\prm = \sum_{k=1}^{\infty} \prw_{k} \delta_{\prl_{k}}$,
	we draw
	$$
		\obsw_{n,k} \indep \bern( \obsw | \prw_{k} ),
	$$
	which is well-defined since every atom weight $\prw_{k}$ of $\prm$ is in $(0,1]$ by the beta process
	construction. Thus,
	$$
		\obsm_{n} = \sum_{k=1}^{\infty} \obsw_{n,k} \delta_{\prl_{k}}.
	$$
	The marginal distribution of the $\obsm_{1:N}$ in this case is often called an
	\emph{Indian buffet process} \citep{griffiths:2006:infinite,thibaux:2007:hierarchical}. The locations of atoms in $\obsm_{n}$ are thought of as the dishes
	sampled by the $n$th customer.
	
	We take a moment to highlight the fact that continuous distributions for $\dlike(d\obsw | \prw)$ are
	precluded based on the Bayesian nonparametric desiderata by considering an alternative
	likelihood. Consider instead if $\dlike(d\obsw | \prw)$ were continuous here. Then $\obsm_{1}$
	would have atoms at every atom of $\prm$. In the Indian buffet process analogy, any customer
	would sample an infinite number of dishes, which contradicts our assumption
	that our data are finite. Indeed, any customer would sample all of the dishes at once.
	 It is quite often the case in practical applications, though, that the $\obsm_{n}$
	are merely latent variables, with the observed variables chosen according to a (potentially continuous) distribution
	given $\obsm_{n}$ \citep{griffiths:2006:infinite,thibaux:2007:hierarchical}; consider, e.g., mixture and admixture models. These cases are not precluded by our development.
	
	Finally, then, we may apply \asu{fin}, which specifies that the number of atoms
	in each observation $\obsm_{n}$ is finite; in this case, the assumption means
	\begin{align*}
		\lefteqn{ \sum_{\obsw=1}^{\infty} \int_{\prw \in \obssp} \wratem(d\prw)
			\cdot \denslike(\obsw | \prw) 
			= \int_{\prw \in (0,1]} \wratem(d\prw)
				\cdot \denslike(1 | \prw) } \\
			& \quad \textrm{since $\prw$ is supported on $(0,1]$ and $\obsw$ is supported on $\{0,1\}$} \\
			&= \int_{\prw \in (0,1]} 
				\massp \prw^{-\discp-1} (1-\prw)^{\concp + \discp - 1} d\prw
				\cdot \prw 
			= \massp \int_{\prw \in (0,1]} \prw^{1-\discp - 1} (1-\prw)^{\concp + \discp - 1} d\prw 
			< \infty.
	\end{align*}
	The integral here is finite if and only if $1-\discp$ and $\concp + \discp$ are the parameters of 
	a proper beta distribution: i.e., if and only if $\discp < 1$ and $\concp > -\discp$.
	Together with the restrictions above, these restrictions
	imply the following allowable parameter ranges
	for the beta process fixed hyperparameters:
	\begin{align} \label{eq:bp_params}
	\begin{split}
		\massp > 0, \quad
		\discp \in [0,1), \quad
		\concp > -\discp, \quad
		\rho_{fix,k}, \sigma_{fix,k} &> 0 \quad \textrm{for all $k \in [K_{fix}]$}.
	\end{split}
	\end{align}
	These correspond to the hyperparameter ranges previously found in \citep{teh:2009:indian, broderick:2012:beta}.
\end{example}

\section{Posteriors} \label{sec:post}

In \mysec{gen}, we defined a full Bayesian model
consisting of a CRM prior for $\prm$ and a CRM likelihood
for an observation $\obsm$ conditional on $\prm$.
Now we would like to calculate the posterior distribution
of $\prm | \obsm$.

\begin{theorem}[Bayesian nonparametric posteriors] \label{thm:post}
	Let $\prm$ be a completely random measure that satisfies \asus{fix} and
	\asuss{inf}; that is, $\prm$ is a CRM with $K_{fix}$ fixed atoms such that
	$K_{fix} < \infty$ and such that the $k$th atom can be written
	$\prw_{fix,k} \delta_{\prl_{fix,k}}$ with
	$$
		\prw_{fix,k} \indep \dfixwk(d\prw)
	$$
	for proper distribution $\dfixwk$ and deterministic $\prl_{fix,k}$.
	Let the ordinary component of $\prm$ have rate measure
	$$
		\fullratem(d\prw \times d\prl) = \wratem(d\prw) \cdot \dordloc(d\prl),
	$$
	where $\dordloc$ is a proper distribution and $\wratem(\obssp) = \infty$.
	Write $\prm = \sum_{k=1}^{\infty} \prw_{k} \delta_{\prl_{k}}$, and let
	$\obsm$ be generated conditional on $\prm$ according to 
	$\obsm = \sum_{k=1}^{\infty} \obsw_{k} \delta_{\prl_{k}}$ with
	$\obsw_{k} \indep \denslike(\obsw | \prw_{k})$ for proper, discrete
	probability mass function $\denslike$. And suppose $\obsm$ and $\prm$ jointly satisfy
	\asu{fin} so that
	$$
		\sum_{\obsw=1}^{\infty} \int_{\prw \in \obssp} \wratem(d\prw) \denslike(\obsw | \prw) < \infty.
	$$
	
	Then let $\prm_{post}$ be a random measure with the distribution of $\prm | \obsm$.
	$\prm_{post}$ is a completely random measure with three parts.
	\begin{enumerate}
		\item For each $k \in [K_{fix}]$, $\prm_{post}$ has a fixed-location atom
		at $\prl_{fix,k}$ with weight $\prw_{post,fix,k}$ distributed according to
		the finite-dimensional posterior $\dfixwpostk(d\prw)$
		that comes from prior $\dfixwk$, likelihood
		$\denslike$, and observation
		$\obsm(\{\prl_{fix,k}\})$.
		\item Let $\{\obsw_{new,k} \delta_{\prl_{new,k}}: k \in [K_{new}] \}$ be the atoms
		of $\obsm$ that are not at fixed locations in the prior of $\prm$. $K_{new}$ is finite
		by \asu{fin}. Then $\prm_{post}$ has a fixed-location atom
		at $\obsw_{new,k}$ with random weight $\prw_{post,new,k}$,
		whose distribution $\dnewwpostk(d\prw)$ is proportional to
		$$
			\wratem(d\prw) \denslike(\obsw_{new,k} | \prw).
		$$
		\item The ordinary component of $\prm_{post}$ has rate measure
		$$
			\wratem_{post}(d\prw) := \wratem(d\prw) \denslike(0 | \prw).
		$$
	\end{enumerate}
\end{theorem}

%
\begin{proof}
To prove the theorem, we consider in turn each of the two parts of the prior: the fixed-location
component and the ordinary component.
First, consider any fixed-location atom, $\prw_{fix,k} \delta_{\prl_{fix,k}}$, in the prior.
All of the other fixed-location atoms in the prior, as well as the prior ordinary component, are
independent of the random weight $\prw_{fix,k}$. So it follows that all of $\obsm$ except
$\obsw_{fix,k} := \obsm(\{\prl_{fix,k}\})$ is independent of $\prw_{fix,k}$. Thus the posterior
has a fixed atom located at $\prl_{fix,k}$ whose weight, which we denote $\prw_{post,fix,k}$,
has distribution
$$
	\dfixwpostk(d\prw) \propto \dfixwk(d\prw) \denslike(\obsw_{fix,k} | \prw),
$$
which follows from the usual finite Bayes Theorem.

Next, consider the ordinary component in the prior. 
Let
$$
	\locsp_{fix} = \{\prl_{fix,1}, \ldots, \prl_{fix,K_{fix}}\}
$$
be the set of fixed-location
atoms in the prior. Recall that $\locsp_{fix}$ is deterministic, and since
$\dordloc$ is continuous, all of the fixed-location atoms and ordinary component
atoms of $\prm$ are at a.s.\ distinct locations. So the measure
$\obsm_{fix}$ defined by
$$
	\obsm_{fix}(A) := \obsm(A \cap \locsp_{fix})
$$
can be derived purely from
$\obsm$, without knowledge of $\prm$. It follows that
the measure $\obsm_{ord}$ defined by
$$
	\obsm_{ord}(A) := \obsm(A \cap (\locsp \backslash \locsp_{fix}))
$$
can be derived purely from $\obsm$ without knowledge of $\prm$. 
$\obsm_{ord}$ is the same as the observed data measure $\obsm$
but with atoms only at atoms of the ordinary
component of $\prm$ and not at the fixed-location atoms of $\prm$.

Now for any value $\obsw \in \Zplus$, let
$$
	\{\prl_{new,\obsw,1}, \ldots, \prl_{new,\obsw,K_{new,\obsw}}\}
$$
be all of the
locations of atoms of size $\obsw$ in $\obsm_{ord}$.
By \asu{fin}, the number of such atoms, $K_{new,\obsw}$, is finite.
Further let
$\prw_{new,\obsw,k} := \prm(\{\prl_{new,\obsw,k}\})$.
Then the values $\{\prw_{new,\obsw,k}\}_{k=1}^{K_{new,\obsw}}$ are generated
from a thinned Poisson point process with rate measure
\begin{equation} \label{eq:wratem_post}
	\wratem_{\obsw}(d\prw) := \wratem(d\prw) \denslike(\obsw | \prw).
\end{equation}
And since $\wratem_{\obsw}(\mathbb{R}_{+}) < \infty$ by assumption,
each $\prw_{new,\obsw,k}$
has distribution equal to the normalized rate measure in \eq{wratem_post}.
Note that $\prw_{new,\obsw,k} \delta_{\prl_{new,\obsw,k}}$ is a fixed-location
atom in the posterior now that its location is known from the observed $\obsm_{ord}$.

By contrast, if a likelihood draw at an ordinary component atom
in the prior returns a zero, that atom is not observed in $\obsm_{ord}$. 
Such atom weights in $\prm_{post}$ thus form a marked Poisson point process with
rate measure
$$
	\wratem(d\prw) \denslike(0 | \prw),
$$
as was to be shown.
\end{proof}

In \thm{post}, we consider generating $\prm$ and then a single
data point $\obsm$ conditional on $\prm$.
Now suppose we generate $\prm$ and then $N$ data points,
$\obsm_{1}, \ldots, \obsm_{N}$, iid conditional
on $\prm$.
In this case, \thm{post} may be iterated to find the posterior
$\prm | \obsm_{1:N}$. In particular, \thm{post} gives the 
ordinary component and fixed atoms of the random measure
$\prm_{1} := \prm | \obsm_{1}$. Then, using $\prm_{1}$ as the prior
measure
and $\obsm_{2}$ as the data point, another application of
\thm{post} gives $\prm_{2} := \prm | \obsm_{1:2}$. We continue
recursively using $\prm | \obsm_{1:n}$ for $n$ between 1 and $N-1$ as the prior measure until
we find $\prm | \obsm_{1:N}$. The result is made explicit in the following corollary.

\begin{corollary}[Bayesian nonparametric posteriors given multiple data points] \label{cor:post}
	Let $\prm$ be a completely random measure that satisfies \asus{fix} and
	\asuss{inf}; that is, $\prm$ is a CRM with $K_{fix}$ fixed atoms such that
	$K_{fix} < \infty$ and such that the $k$th atom can be written
	$\prw_{fix,k} \delta_{\prl_{fix,k}}$ with
	$$
		\prw_{fix,k} \indep \dfixwk(d\prw)
	$$
	for proper distribution $\dfixwk$ and deterministic $\prl_{fix,k}$.
	Let the ordinary component of $\prm$ have rate measure
	$$
		\fullratem(d\prw \times d\prl) = \wratem(d\prw) \cdot \dordloc(d\prl),
	$$
	where $\dordloc$ is a proper distribution and $\wratem(\obssp) = \infty$.
	Write $\prm = \sum_{k=1}^{\infty} \prw_{k} \delta_{\prl_{k}}$, and let
	$\obsm_1, \ldots, \obsm_n$ be generated conditional on $\prm$ according to 
	$\obsm = \sum_{k=1}^{\infty} \obsw_{n,k} \delta_{\prl_{n,k}}$ with
	$\obsw_{n,k} \indep \denslike(\obsw | \prw_{k})$ for proper, discrete
	probability mass function $\denslike$. And suppose $\obsm_1$ and $\prm$ jointly satisfy
	\asu{fin} so that
	$$
		\sum_{\obsw=1}^{\infty} \int_{\prw \in \obssp} \wratem(d\prw) \denslike(\obsw | \prw) < \infty.
	$$
	It is enough to make the assumption for $\obsm_1$ since the $\obsm_n$ are iid conditional on $\prm$.
	
	Then let $\prm_{post}$ be a random measure with the distribution of $\prm | \obsm_{1:N}$.
	$\prm_{post}$ is a completely random measure with three parts.
	\begin{enumerate}
		\item For each $k \in [K_{fix}]$, $\prm_{post}$ has a fixed-location atom
		at $\prl_{fix,k}$ with weight $\prw_{post,fix,k}$ distributed according to
		the finite-dimensional posterior $\dfixwpostk(d\prw)$
		that comes from prior $\dfixwk$, likelihood
		$\denslike$, and observation
		$\obsm(\{\prl_{fix,k}\})$.
		\item Let $\{\prl_{new,k}: k \in [K_{new}] \}$ be the union of atom
		locations across $\obsm_{1}, \obsm_{2}, \ldots, \obsm_{N}$
		minus the fixed locations in the prior of $\prm$.
		$K_{new}$ is finite.
		Let $\obsw_{new,n,k}$ be the weight of the atom in $\obsm_{n}$
		located at $\prl_{new,k}$. Note that at least one of $\obsw_{new,n,k}$
		across $n$ must be non-zero, but in general $\obsw_{new,n,k}$ may 
		equal zero.
		Then $\prm_{post}$ has a fixed-location atom
		at $\obsw_{new,k}$ with random weight $\prw_{post,new,k}$,
		whose distribution $\dnewwpostk(d\prw)$ is proportional to
		$$
			\wratem(d\prw) \prod_{n=1}^{N} \denslike(\obsw_{new,n,k} | \prw).
		$$
		\item The ordinary component of $\prm_{post}$ has rate measure
		$$
			\wratem_{post,n}(d\prw) := \wratem(d\prw) \left[ \denslike(0 | \prw) \right]^{n}.
		$$
	\end{enumerate}
\end{corollary}

\begin{proof}
	\cor{post} follows from recursive application of \thm{post}. In order to recursively apply \thm{post},
	we need to verify that \asus{fix}, \asuss{inf}, and \asuss{fin} hold for the posterior
	$\prm | \obsm_{1:(n+1)}$ when they
	hold for the prior $\prm | \obsm_{1:n}$.
	Note that the number of fixed atoms in the posterior is the number of fixed
	atoms in the prior plus the number of new atoms in the posterior. By \thm{post}, these
	counts
	are both finite as long as $\prm | \obsm_{1:n}$ satisfies \asus{fix} and \asuss{fin},
	which both hold for $n=0$ by assumption and $n > 0$ by the recursive assumption.
	So \asu{fix} holds for $\prm | \obsm_{1:(n+1)}$.
	
	Next we notice that since \asu{inf} implies that there 
	is an infinite number of ordinary component atoms in $\prm | \obsm_{1:n}$ and only finitely
	many become fixed atoms in the posterior by \asu{fin}, it must be that $\prm | \obsm_{1:(n+1)}$
	has infinitely many ordinary component atoms. So \asu{inf} holds for $\prm | \obsm_{1:(n+1)}$.
	
	Finally, we note that
	\begin{align*}
		\lefteqn{ \sum_{\obsw=1}^{\infty} \int_{\prw \in \obssp} \wratem_{post,n}(d\prw) \denslike(\obsw | \prw) } \\
			&= \sum_{\obsw=1}^{\infty} \int_{\prw \in \obssp} \wratem(d\prw) \left[ \denslike(0 | \prw) \right]^{n} \denslike(\obsw | \prw) 
			\le \sum_{\obsw=1}^{\infty} \int_{\prw \in \obssp} \wratem(d\prw) \denslike(\obsw | \prw)
			< \infty,
	\end{align*}
	where the penultimate inequality follows since $\denslike(0 | \prw) \in [0,1]$ and
	where the inequality follows by \asu{fin} on the original $\prm$ (conditioned on no data). So
	\asu{fin} holds for $\prm | \obsm_{1:(n+1)}$.
\end{proof}

We now illustrate the results of the theorem with an example.

\begin{example} \label{ex:bp_nbp}
	Suppose we again start with a beta process prior for $\prm$
	as in \ex{bp_bep}. This time we consider a \emph{negative binomial
	process likelihood} \citep{zhou:2012:beta,broderick:2015:combinatorial}.
	The negative binomial process specifies that,
	given $\prm = \sum_{k=1}^{\infty} \prw_{k} \delta_{\prl_{k}}$,
	we draw $\obsm = \sum_{k=1}^{\infty} \obsw_{k} \delta_{\prl_{k}}$ with
	$$
		\obsw_{k} \indep \negbin( \obsw | r, \prw_{k} ),
	$$
	for some fixed hyperparameter $r > 0$. 
	So
	$$
		\obsm_{n} = \sum_{k=1}^{\infty} \obsw_{n,k} \delta_{\prl_{k}}.
	$$
	In this case, \asu{fin}
	translates into the following restriction.
	\begin{align*}
		\lefteqn{ \sum_{\obsw=1}^{\infty} \int_{\prw \in \obssp} \wratem(d\prw)
			\cdot \denslike(\obsw | \prw) } \\
			&= \int_{\prw \in \obssp}
				\wratem(d\prw)
				\cdot
				\left[ 1 - \denslike(0 | \prw) \right] 
			=\int_{\prw \in (0,1]}
				\massp \prw^{-\discp-1} (1-\prw)^{\concp + \discp - 1} d\prw
				\cdot
				\left[ 1 - (1-\prw)^{r} \right] 
			< \infty,
	\end{align*}
	where the penultimate equality follows since the support of $\wratem(d\prw)$ is $(0,1]$.
	
	By a Taylor expansion, we have $1 - (1-\prw)^{r} = r \prw + o(\prw)$ as $\prw \rightarrow 0$, 
	so we require
	$$
		\int_{\prw \in (0,1]}
				\prw^{1-\discp-1} (1-\prw)^{\concp + \discp - 1} d\prw
			< \infty,
	$$
	which is satisfied if and only if $1-\discp$ and $\concp + \discp$ are the parameters
	of a proper beta distribution. Thus, we have the same parameter restrictions as in
	\eq{bp_params}.

	Now we calculate the posterior given the beta process prior on $\prm$ and
	the negative binomial process likelihood for $\obsm$ conditional on $\prm$.
	In particular, the posterior has the distribution of $\prm_{post}$, a CRM
	with three parts given by \thm{post}.
	
	First, at each fixed atom $\prl_{fix,k}$ of the prior
	with weight $\prw_{fix,k}$ given by \eq{bp_fix},
	there is a fixed atom in the posterior
	with weight $\prw_{post,fix,k}$. Let 
	$\obsw_{post,fix,k} := \obsm(\{\prl_{fix,k}\})$.
	Then $\prw_{post,fix,k}$ has distribution
	\begin{align} \label{eq:bp_bnp_fix_old}
	\begin{split}
		\dfixwpostk(d\prw)
			&\propto F_{fix}(d\prw) \cdot \denslike(\obsw_{post,fix,k} | \prw) \\
			&= \tb(\prw | \rho_{fix,k}, \sigma_{fix,k}) \; d\prw
				\cdot
				\negbin(\obsw_{post,fix,k} | r, \prw ) \\
			&\propto \prw^{\rho_{fix,k} - 1} (1 - \prw)^{\sigma_{fix,k} - 1} \; d\prw
				\cdot
				\prw^{\obsw_{post,fix,k}} (1 - \prw)^{r} \\
			&\propto \tb\left( \prw \left\vert \rho_{fix,k} + \obsw_{post,fix,k},
				\sigma_{fix,k} + r \right. \right) \; d\prw.
	\end{split}
	\end{align}
		
	Second, for any atom $\obsw_{new,k} \delta_{\prl_{new,k}}$ in $\obsm$
	that is not at a fixed location in the prior, $\prm_{post}$ has a fixed atom 
	at $\prl_{new,k}$ whose weight $\prw_{post,new,k}$ has distribution
	\begin{align} \label{eq:bp_bnp_fix_new}
	\begin{split}
		\dnewwpostk(d\prw)
			&\propto \wratem(d\prw) \cdot \denslike(\obsw_{new,k} | \prw) \\
			&= \wratem(d\prw) \cdot \negbin(\obsw_{new,k} | r, \prw ) \\
			&\propto \prw^{-\discp - 1} (1-\prw)^{\concp + \discp - 1} \; d\prw
				\cdot \prw^{\obsw_{new,k}} (1 - \prw)^{r} \\
			&\propto \tb\left( \prw \left\vert -\discp + \obsw_{new,k}, \concp + \discp + r \right. \right) \; d\prw,
	\end{split}
	\end{align}
	which is a proper distribution since we have the following
	restrictions on its parameters. For one, by assumption,
	$\obsw_{new,k} \ge 1$.
	And further, by
	\eq{bp_params}, we have $\discp \in [0,1)$ as well
	as $\concp + \discp > 0$ and $r > 0$.
	
	Third, the ordinary component of $\prm_{post}$ has rate measure
	\begin{align*}
		\wratem(d\prw) \denslike(0 | \prw)
			= \massp \prw^{-\discp - 1} (1-\prw)^{\concp + \discp - 1} \; d\prw
				\cdot (1-\prw)^{r} 
			= \massp \prw^{-\discp - 1} (1-\prw)^{\concp + r + \discp - 1} \; d\prw.
	\end{align*}
	
	Not only have we found the posterior distribution $\prm_{post}$ above, but
	now we can note that the posterior is in the same form as the prior with updated 
	ordinary component hyperparameters:
	\begin{align*}
		\massp_{post} = \massp, \quad 
		\discp_{post} = \discp, \quad
		\concp_{post} = \concp + r.
	\end{align*}
	The posterior also has old and new beta-distributed fixed atoms with beta distribution
	hyperparameters given in \eq{bp_bnp_fix_old}
	and \eq{bp_bnp_fix_new}, respectively.
	Thus, we have proven that the beta process is, in fact,
	conjugate to the negative binomial process. An alternative proof
	was first given by \citet{broderick:2015:combinatorial}.
\end{example}

As in \ex{bp_nbp}, we can use \thm{post} not only to calculate posteriors but also, once those posteriors
are calculated, to check for conjugacy. This approach unifies existing disparate approaches to
Bayesian nonparametric conjugacy.
However, it still requires the practitioner to guess the right conjugate prior for a given likelihood.
In the next section, we define a notion of exponential families
for CRMs, and we show how to automatically 
construct a conjugate prior for any exponential family likelihood.

\section{Exponential families} \label{sec:exp_fam}

Exponential families are what typically make conjugacy so powerful
in the finite case. For one, when a finite likelihood belongs to an exponential family,
then existing results give an automatic conjugate, exponential family prior for that likelihood.
In this section, we review finite exponential families, define
\emph{exponential CRMs}, and show that analogous automatic conjugacy results
can be obtained for exponential CRMs. Our development of exponential
CRMs will also allow particularly straightforward results for
size-biased representations (\cor{size} in \mysec{size}) and 
marginal processes (\cor{marg} in \mysec{marg}).

In the finite-dimensional case, suppose we have some (random) parameter $\prw$
and some (random) observation $\obsw$ whose distribution is conditioned on $\prw$.
We say the distribution $\dexplike$ of $\obsw$ conditional on $\prw$ is in an exponential family if
\begin{align} \label{eq:exp_fam_like}
\begin{split}
	\dexplike(d\obsw | \prw)
		= \densexplike(\obsw | \prw) \; d\obsw
		= \basem(\obsw) \exp\left\{ \langle \natpar(\prw), \suffstat(\obsw) \rangle - \logpart(\prw) \right\} \; \expm(d\obsw),
\end{split}
\end{align}
where $\natpar(\prw)$ is the \emph{natural parameter}, $\suffstat(\obsw)$
is the \emph{sufficient statistic}, $\basem(\obsw)$ is the \emph{base density},
and $\logpart(\prw)$ is the \emph{log partition function}. We denote the
density of $\dexplike$ here, which exists by definition, by $\densexplike$.
The measure $\expm$---with
respect to which the density $\densexplike$ exists---is typically
Lebesgue measure when $\dexplike$ is diffuse or counting measure when
$\dexplike$ is atomic.
 $\logpart(\prw)$ is determined
by the condition that $\dexplike(d\obsw | \prw)$ have unit total mass on
its support.

It is a classic result \citep{diaconis:1979:conjugate}
that the following distribution for $\prw \in \mathbb{R}^{D}$
constitutes a conjugate
prior:
\begin{align} \label{eq:exp_fam_prior}
\begin{split}
	\dexpprior(d\prw)
		= \densexpprior(\prw) \; d\prw
		= \exp\left\{ \langle \xi, \natpar(\prw) \rangle + \lambda \left[ -\logpart(\prw) \right] 
				- \priorlogpart(\xi, \lambda)
			\right\} \; d\prw
			.
\end{split}
\end{align}
$\dexpprior$ is another exponential family distribution, now with natural parameter
$(\xi', \lambda)'$, sufficient statistic $(\natpar(\prw)', -\logpart(\prw))'$,
and log partition function $B(\xi, \lambda)$.
Note that the logarithms of the densities in both \eq{exp_fam_like} and 
\eq{exp_fam_prior} are linear in $\natpar(\prw)$ and $-\logpart(\prw)$.
So, by Bayes Theorem,
the posterior $\dexppost$ also has these quantities as sufficient statistics in $\prw$,
and we can see $\dexppost$ must have the following form.
\begin{align} \label{eq:exp_fam_post}
\begin{split}
	\lefteqn{ \dexppost(d\prw | \obsw)
		= \densexppost(\prw | \obsw) \; d\prw } \\
		&= \exp\left\{ \langle \xi + \suffstat(\obsw), \natpar(\prw) \rangle
		+ (\lambda + 1) \left[ -\logpart(\prw) \right] 
				- B(\xi + \suffstat(\obsw), \lambda + 1)
			\right\} \; d\prw.
\end{split}
\end{align}
Thus we see that $\dexppost$ belongs to the same exponential
family as $\dexpprior$ in \eq{exp_fam_prior}, and 
hence $\dexpprior$ is a conjugate prior for 
$\dexplike$ in \eq{exp_fam_like}.

\subsection{Exponential families for completely random measures} \label{sec:exp_fam_crm}

In the finite-dimensional case, we saw that for any exponential family likelihood,
as in \eq{exp_fam_like}, we can always construct a conjugate
exponential family prior, given by \eq{exp_fam_prior}.

In order to prove a similar result for CRMs, we start by defining a notion
of exponential families for CRMs.
\begin{definition} \label{def:exp_fam_crm}
	We say that a CRM $\prm$ is an \emph{exponential CRM}
	if it has the following two parts. First, let $\prm$ have $K_{fix}$ fixed-location
	atoms, where $K_{fix}$ may be finite or infinite. The $k$th fixed-location atom
	is located at any $\prl_{fix,k}$, unique from the other fixed locations,
	and has random weight $\prw_{fix,k}$, whose distribution has density
	$\densfixwk$:
	$$
		\densfixwk(\prw)
			= \basem(\prw) \exp\left\{ \langle \natpar(\zeta_{k}), \suffstat(\prw) \rangle 
				- \logpart(\zeta_{k}) \right\},
	$$
	for some base density $\basem$, natural parameter function $\natpar$,
	sufficient statistic $\suffstat$, and log partition function
	$\logpart$ shared across atoms. Here, $\zeta_{k}$ is an
	atom-specific parameter.
	
	Second, let $\prm$ have an ordinary component with rate measure
	$\fullratem(d\prw \times d\prl) = \wratem(d\prw) \cdot \dordloc(d\prl)$
	for some proper distribution $\dordloc$ and weight rate measure $\wratem$
	of the form
	$$
		\wratem(d\prw) = \massp \exp\left\{ \langle \natpar(\zeta), \suffstat(\prw) \rangle \right\}.
	$$
	In particular, $\natpar$ and $\suffstat$ are shared with the fixed-location atoms, and fixed hyperparameters
	$\massp$ and $\zeta$ are unique to the ordinary component.
\end{definition}

\subsection{Automatic conjugacy for completely random measures} \label{sec:auto_conj}

With \mydef{exp_fam_crm} in hand, we can specify an automatic Bayesian
nonparametric conjugate prior 
for an exponential CRM likelihood.

\begin{theorem}[Automatic conjugacy] \label{thm:auto_conj}
	Let $\prm = \sum_{k=1}^{\infty} \prw_{k} \delta_{\prl_{k}}$,
	in accordance with \asu{inf}.
	Let $\obsm$ be generated conditional on $\prm$
	according to an exponential CRM with 
	fixed-location atoms at $\{\prl_{k}\}_{k=1}^{\infty}$
	and no ordinary component. 
	In particular, the distribution of the weight $\obsw_{k}$
	at $\prl_{k}$
	of $\obsm$
	has the following density conditional on the weight
	$\prw_{k}$ at $\prl_{k}$ of $\prm$:
	$$
		\denslike(\obsw | \prw_{k})
			= \basem(\obsw) \exp\left\{ \langle \natpar(\prw_{k}), \suffstat(\obsw) \rangle 
				- \logpart(\prw_{k}) \right\}.
	$$
	
	Then a conjugate prior for $\prm$
	is the following exponential CRM distribution.
	First, let $\prm$ have $K_{prior,fix}$ fixed-location atoms,
	in accordance with \asu{fix}. The $k$th such
	atom has random weight $\prw_{fix,k}$ with proper density
	$$
		\densfixwpriork(\prw)
			= \exp\left\{ \langle \xi_{fix,k}, \natpar(\prw) \rangle + \lambda_{fix,k} 
				\left[ -\logpart(\prw) \right] - \priorlogpart(\xi_{fix,k}, \lambda_{fix,k}) \right\},
	$$
	where $(\eta', -A)'$ here is the sufficient statistic and $B$ is the
	log partition function. $\xi_{fix,k}$ and $\lambda_{fix,k}$ are fixed hyperparameters
	for this atom weight.
	
	Second, let $\prm$ have ordinary component characterized by any proper
	distribution $\dordloc$ and weight rate measure
	$$
		\wratem(d\prw) = \massp \exp\left\{ \langle \xi, \natpar(\prw) \rangle
			+ \lambda \left[ -\logpart(\prw) \right] \right\},
	$$
	where $\massp$, $\xi$, and $\lambda$ are fixed hyperparameters
	of the weight rate measure chosen to satisfy \asus{inf} and \asuss{fin}.
\end{theorem}

\begin{proof}
To prove the conjugacy of the prior for $\prm$ with the likelihood for $\obsm$,
we calculate the posterior distribution of $\prm | \obsm$ using \thm{post}. Let $\prm_{post}$
be a CRM with the distribution of $\prm | \obsm$. Then, by \thm{post}, $\prm_{post}$ has
the following three parts.

First, at any fixed location $\prl_{fix,k}$ in the prior, let $\obsw_{fix,k}$
be the value of $\obsm$ at that location. Then $\prm_{post}$ has a fixed-location
atom at $\prl_{fix,k}$, and its weight $\prw_{post,fix,k}$ has distribution 
\begin{align*}
	\dfixwpostk(d\prw) 
		&\propto \densfixwpriork(\prw) \; d\prw \cdot \denslike(\obsw_{fix,k} | \prw) \\
		&\propto \exp\left\{ \langle \xi_{fix,k}, \natpar(\prw) \rangle + \lambda_{fix,k} 
				\left[ -\logpart(\prw) \right] \right\} \; d\prw
			\cdot
				\exp\left\{ \langle \natpar(\prw), \suffstat(\obsw_{fix,k}) \rangle 
				- \logpart(\prw) \right\} \; d\prw \\
		&= \exp\left\{ \langle \xi_{fix,k} + \suffstat(\obsw_{fix,k}), \natpar(\prw) \rangle
			+ (\lambda_{fix,k} + 1) \left[ -\logpart(\prw) \right] \right\} \; d\prw.
\end{align*}
It follows, from putting in the normalizing constant,
that the distribution of $\prw_{post,fix,k}$ has density
\begin{align*}
	\densfixwpostk(\prw)
		&= \exp\left\{ \langle \xi_{fix,k} + \suffstat(\obsw_{fix,k}), \natpar(\prw) \rangle
			+ (\lambda_{fix,k} + 1) \left[ -\logpart(\prw) \right] \right. \\
		& \left. {} 
			- \priorlogpart(\xi_{fix,k} + \suffstat(\obsw_{fix,k}), \lambda_{fix,k} + 1) \right\}.
\end{align*}

Second, for any atom $\obsw_{new,k} \delta_{\prl_{new,k}}$ in $\obsm$
that is not at a fixed location in the prior, $\prm_{post}$ has a fixed
atom at $\prl_{new,k}$ whose weight $\prw_{post,new,k}$ has distribution
\begin{align*}
	\dnewwpostk(\prw)
		&\propto \wratem(d\prw) \cdot \denslike(\obsw_{new,k} | \prw) \\
		&\propto \exp\left\{ \langle \xi, \natpar(\prw) \rangle
				+ \lambda \left[ -\logpart(\prw) \right] \right\}
			\cdot
			\exp\left\{ \langle \natpar(\prw), \suffstat(\obsw_{new,k}) \rangle 
				- \logpart(\prw) \right\} \; d\prw \\
		&= \exp\left\{ \langle \xi + \suffstat(\obsw_{new,k}), \natpar(\prw) \rangle
			+ (\lambda + 1) \left[ -\logpart(\prw) \right] \right\} \; d\prw
\end{align*}
and hence density
\begin{align*}
	\densnewwpostk(\prw)
		&= \exp\left\{ \langle \xi + \suffstat(\obsw_{new,k}), \natpar(\prw) \rangle
			+ (\lambda + 1) \left[ -\logpart(\prw) \right] 
			- \priorlogpart(\xi + \suffstat(\obsw_{new,k}), \lambda + 1) \right\}.
\end{align*}

Third, the ordinary component of $\prm_{post}$ has weight rate measure
\begin{align*}
	\lefteqn{ \wratem(d\prw) \cdot \denslike(0 | \prw) } \\
		&= \massp \exp\left\{ \langle \xi, \natpar(\prw) \rangle
				+ \lambda \left[ -\logpart(\prw) \right] \right\}
			\cdot 
			\basem(0) \exp\left\{ \langle \natpar(\prw), \suffstat(0) \rangle 
				- \logpart(\prw) \right\} \\
		&= \massp \basem(0) \cdot \exp\left\{ \langle \xi + \suffstat(0), \natpar(\prw) \rangle
				+ (\lambda + 1) \left[ -\logpart(\prw) \right] \right\}.
\end{align*}

Thus, the posterior rate measure is in the same exponential CRM form
as the prior rate measure with
updated hyperparameters:
\begin{align*}
	\massp_{post} = \massp \basem(0), \quad
	\xi_{post} = \xi + \suffstat(0), \quad
	\lambda_{post} = \lambda + 1.
\end{align*}
Since we see that the posterior fixed-location atoms are likewise in the same
exponential CRM form as the prior,
we have shown that conjugacy holds, as desired.
\end{proof}

We next use \thm{auto_conj} to give proofs of conjugacy
in cases where conjugacy has not previously been established in the Bayesian
nonparametrics literature. 

\begin{example} \label{ex:plp_gp_auto_conj}
	Let $\obsm$ be generated according to a \emph{Poisson likelihood
	process}\footnote{We use the term ``Poisson likelihood process''
	to distinguish this specific Bayesian nonparametric likelihood from
	the Poisson point process.}
	conditional on $\prm$. That is,
	$
		\obsm = \sum_{k=1}^{\infty} \obsw_{k} \delta_{\prl_{k}}
	$
	conditional on
	$\prm = \sum_{k=1}^{\infty} \prw_{k} \delta_{\prl_{k}}$
	has an exponential CRM distribution with only a
	fixed-location component.
	The weight $\obsw_{k}$ at location $\prl_{k}$
	has support on $\Zstar$ and has a Poisson density
	with parameter $\prw_{k} \in \obssp$:
	\begin{align} \label{eq:plp_dens}
	\begin{split}
		\denslike(\obsw | \prw_{k})
			= \frac{1}{\obsw !} \prw_{k}^{\obsw} e^{-\prw_{k}}
			= \frac{1}{\obsw !} \exp\left\{ \obsw \log(\prw_{k}) - \prw_{k} \right\}
			.
	\end{split}
	\end{align}
	The final line is rewritten to emphasize the exponential family form
	of this density, with
	\begin{align*}
		\basem(\obsw) = \frac{1}{\obsw !}, \quad
		\suffstat(\obsw) = \obsw, \quad
		\natpar(\prw) = \log(\prw), \quad
		\logpart(\prw) = \prw.
	\end{align*}
	By \thm{auto_conj}, this Poisson likelihood process has a Bayesian nonparametric
	conjugate prior for $\prm$ with two parts.
	
	First, $\prm$ has a set of $K_{prior,fix}$ fixed-location atoms,
	where $K_{prior,fix} < \infty$ by \asu{fix}. The $k$th
	such atom has random weight $\prw_{fix,k}$ with density
	\begin{align}
		\densfixwpriork(\prw)
			\nonumber
			&= \exp\left\{ \langle \xi_{fix,k}, \natpar(\prw) \rangle + \lambda_{fix,k} 
				\left[ -\logpart(\prw) \right] - \priorlogpart(\xi_{fix,k}, \lambda_{fix,k}) \right\} \\
			\nonumber
			&= \prw^{\xi_{fix,k}} e^{-\lambda_{fix,k} \prw}
				\exp\left\{- \priorlogpart(\xi_{fix,k}, \lambda_{fix,k}) \right\} \\
			\label{eq:gp_fixed_atom}
			&= \ga(\prw \left\vert \xi_{fix,k} + 1, \lambda_{fix,k} \right. ),
	\end{align}
	where $\ga(\prw | a, b)$ denotes the gamma density with shape parameter
	$a > 0$ and rate parameter $b > 0$. So we must have fixed hyperparameters
	$\xi_{fix,k} > -1$ and $\lambda_{fix,k} > 0$.
	Further,
	$$
		\exp\left\{- \priorlogpart(\xi_{fix,k}, \lambda_{fix,k}) \right\}
			= \lambda_{fix,k}^{\xi_{fix,k} + 1} / \Gamma(\xi_{fix,k} + 1)
	$$
	to ensure normalization.
	
	Second, $\prm$ has an ordinary component characterized by any proper distribution
	$\dordloc$ and weight rate measure
	\begin{align}
		\wratem(d\prw)
			= \massp \exp\left\{ \langle \xi, \natpar(\prw) \rangle
				+ \lambda \left[ -\logpart(\prw) \right] \right\} \; d\prw
			\label{eq:gp_ord}
			= \massp \prw^{\xi} e^{-\lambda \prw} \; d\prw.
	\end{align}
	Note that \thm{auto_conj} guarantees that the weight rate measure
	will have the same distributional kernel in $\prw$ as the fixed-location atoms.
	
	Finally, we need to choose the allowable hyperparameter ranges for $\massp$,
	$\xi$, and $\lambda$. First, $\massp > 0$ to ensure $\wratem$ is a measure.
	By \asu{inf}, we must have $\wratem(\obssp) = \infty$, so $\wratem$ 
	must represent an improper gamma distribution. As such, we require
	either $\xi + 1 \le 0$ or $\lambda \le 0$.
	By \asu{fin}, we must have
	\begin{align*}
		\sum_{\obsw=1}^{\infty} \int_{\prw \in \obssp}
			\wratem(d\prw) \cdot \denslike(\obsw | \prw) 
			= \int_{\prw \in \obssp} \wratem(d\prw) \cdot [1 - \denslike(0 | \prw)]
			= \int_{\prw \in \obssp}
				\massp \prw^{\xi} e^{-\lambda \prw} \; d\prw
				\cdot \left[1 - e^{-\prw} \right]
			< \infty.
	\end{align*}
	To ensure the integral over $[1,\infty)$ is finite, we must have $\lambda > 0$. To ensure
	the integral over $(0,1)$ is finite, we note that $1 - e^{-\prw} = \prw + o(\prw)$
	as $\prw \rightarrow 0$. So we require
	$$
		\int_{\prw \in (0,1)}
				\massp \prw^{\xi + 1} e^{-\lambda \prw} \; d\prw
			< \infty,
	$$
	which is satisfied if and only if $\xi + 2 > 0$.
	
	Finally, then the hyperparameter restrictions can be summarized as:
	\begin{align*}
		\massp > 0, \quad
		\xi \in (-2,-1], \quad
		\lambda > 0; \quad
		\xi_{fix,k} > -1 \textrm{ and } \lambda_{fix,k} > 0
			\quad \textrm{for all $k \in [K_{prior,fix}]$}.
	\end{align*}
	
	The ordinary component of the conjugate prior for $\prm$ discovered in this example
	is typically called a \emph{gamma process}. Here, we have for the first time specified
	the distribution of the fixed-location atoms of the gamma process and, also for the first
	time, proved that the gamma process is conjugate to the 
	Poisson likelihood process. 
	We highlight this result as a corollary to \thm{auto_conj}.
	
	\begin{corollary}
	Let the Poisson likelihood process be a CRM with fixed-location
	atom weight distributions as in \eq{plp_dens}.
	Let the gamma process be a CRM with fixed-location atom weight distributions
	as in \eq{gp_fixed_atom} and ordinary component weight measure as in
	\eq{gp_ord}. Then the gamma process is a conjugate Bayesian nonparametric
	prior for the Poisson
	likelihood process.
	\end{corollary}
\end{example}

\begin{example}
	Next, let $\obsm$ be generated according to a new process we call
	an \emph{odds Bernoulli process}. We have previously seen a typical
	Bernoulli process likelihood in \ex{bp_bep}. In the odds Bernoulli process,
	we say that $\obsm$, conditional on $\prm$, has an exponential CRM
	distribution.
	In this case, the weight of the $k$th atom, $\obsw_{k}$, conditional on
	$\prw_{k}$
	has support on $\{0,1\}$ and has a Bernoulli density
	with odds parameter $\prw_{k} \in \obssp$:
	\begin{align} \label{eq:bern_exp}
	\begin{split}
		\denslike(\obsw | \prw_{k})
			&= \prw_{k}^{\obsw} (1 + \prw_{k})^{-1} \\
			&= \exp\left\{ \obsw \log(\prw_{k}) - \log(1 + \prw_{k}) \right\}
			.
	\end{split}
	\end{align}
	That is, if $\rho$ is the probability of a successful Bernoulli draw,
	then $\prw = \rho/(1 - \rho)$ represents the odds ratio of the probability of
	success over the probability of failure. 
	
	The final line of \eq{bern_exp}
	is written to emphasize the exponential family form of this density, with
	\begin{align*}
		\basem(\obsw) = 1, \quad
		\suffstat(\obsw) = \obsw, \quad
		\natpar(\prw) = \log(\prw), \quad
		\logpart(\prw) = \log(1 + \prw).
	\end{align*}
	By \thm{auto_conj}, the likelihood for $\obsm$ has a Bayesian nonparametric conjugate
	prior for $\prm$.
	This conjugate prior has two parts.
	
	First, $\prm$ has a set of $K_{prior,fix}$ fixed-location atoms. The $k$th
	such atom has random weight $\prw_{fix,k}$ with density
	\begin{align}
		\densfixwpriork(\prw)
			\nonumber
			&= \exp\left\{ \langle \xi_{fix,k}, \natpar(\prw) \rangle + \lambda_{fix,k} 
				\left[ -\logpart(\prw) \right] - \priorlogpart(\xi_{fix,k}, \lambda_{fix,k}) \right\} \\
			\nonumber
			&= \prw^{\xi_{fix,k}} (1+\prw)^{-\lambda_{fix,k}}
				\exp\left\{- \priorlogpart(\xi_{fix,k}, \lambda_{fix,k}) \right\} \\
			\label{eq:bp_fixed_atom}
			&= \tbp\left(\prw \left \vert \xi_{fix,k} + 1, \lambda_{fix,k} - \xi_{fix,k} - 1 \right. \right),
	\end{align}
	where $\tbp(\prw | a, b)$ denotes the beta prime density with shape
	parameters $a > 0$ and $b > 0$.
	Further,
	$$
		\exp\left\{- \priorlogpart(\xi_{fix,k}, \lambda_{fix,k}) \right\}
			= \frac{
					\Gamma(\lambda_{fix,k})
				}{
					\Gamma(\xi_{fix,k} + 1)
					\Gamma(\lambda_{fix,k} - \xi_{fix,k} - 1)
				}
	$$
	to ensure normalization.
	
	Second, $\prm$ has an ordinary component characterized by any proper distribution
	$\dordloc$ and weight rate measure
	\begin{align}
		\wratem(d\prw)
			= \massp \exp\left\{ \langle \xi, \natpar(\prw) \rangle
				+ \lambda \left[ -\logpart(\prw) \right] \right\} \; d\prw
			\label{eq:bp_ord_meas}
			= \massp \prw^{\xi} (1+\prw)^{-\lambda} \; d\prw.
	\end{align}
	
	We need to choose the allowable hyperparameter ranges for $\massp$,
	$\xi$, and $\lambda$. First, $\massp > 0$ to ensure $\wratem$ is a measure.
	By \asu{inf}, we must have $\wratem(\obssp) = \infty$, so $\wratem$ 
	must represent an improper beta prime distribution. As such, we require
	either $\xi + 1 \le 0$ or $\lambda - \xi - 1 \le 0$.
	By \asu{fin}, we must have
	\begin{align*}
		\lefteqn{ \sum_{\obsw=1}^{\infty} \int_{\prw \in \obssp}
			\wratem(d\prw) \cdot \denslike(\obsw | \prw) 
			= \int_{\prw \in \obssp} \wratem(d\prw) \cdot \denslike(1 | \prw) } \\
			& \quad \textrm{since the support of $\obsw$ is $\{0,1\}$} \\
			&= \int_{\prw \in \obssp}
				\massp \prw^{\xi} (1+\prw)^{-\lambda} \; d\prw
				\cdot
				\prw^{1} (1 + \prw)^{-1}
			= \massp \int_{\prw \in \obssp}
				\prw^{\xi + 1} (1+\prw)^{-\lambda - 1} \; d\prw
			< \infty.
	\end{align*}
	Since the integrand is the kernel of a beta prime distribution, we simply require
	that this distribution be proper; i.e., $\xi + 2 > 0$ and $\lambda - \xi - 1 > 0$.
	
	The hyperparameter restrictions can be summarized as:
	\begin{align*}
		\massp > 0,
		\xi \in (-2,-1],
		\lambda > \xi + 1; 
		\xi_{fix,k} > -1 \textrm{ and } \lambda_{fix,k} > \xi_{fix,k} + 1
			\textrm{ for all $k \in [K_{prior,fix}]$}.
	\end{align*}

	We call the distribution for $\prm$ described in this example the 
	\emph{beta prime process}. Its ordinary component has previously been defined by
	\citet{broderick:2015:combinatorial}.
	But this result represents the first time the beta prime process
	is described in full, including parameter
	restrictions and fixed-location atoms, as well as the first proof
	of its conjugacy with the odds Bernoulli process. 
	We highlight the latter result as a corollary to \thm{auto_conj} below.
	
	\begin{corollary}
	Let the odds Bernoulli process be a CRM with fixed-location
	atom weight distributions as in \eq{bern_exp}.
	Let the beta process be a CRM with fixed-location atom weight distributions
	as in \eq{bp_fixed_atom} and ordinary component weight measure as in
	\eq{bp_ord_meas}. Then the beta process is a conjugate Bayesian nonparametric
	prior for the odds Bernoulli process.
	\end{corollary}
\end{example}

\section{Size-biased representations} \label{sec:size}

We have shown in \mysec{auto_conj} that 
our exponential CRM (\mydef{exp_fam_crm}) is useful in that we
can find an automatic
Bayesian nonparametric conjugate prior given an exponential
CRM likelihood. We will see in this section and the next
that exponential CRMs allow us to build representations
that allow tractable inference despite the infinite-dimensional
nature of the models we are using.

The best-known size-biased representation of a random measure
in Bayesian nonparametrics is the \emph{stick-breaking} representation
of the Dirichlet process $\prm_{DP}$ \citep{sethuraman:1994:constructive}:
\begin{align} \label{eq:dp_stick}
\begin{split}
	\prm_{DP} &= \sum_{k=1}^{\infty} \prw_{DP, k} \delta_{\prl_{k}}; \\
	\textrm{ For $k \in \Zstar$, }
		\prw_{DP,k} &= \beta_{k} \prod_{j=1}^{k-1} (1-\beta_{j}), \quad
		\beta_{k} \iid \tb(1, \concp), \quad
		\prl_{k} \iid \dordloc,
\end{split}
\end{align}
where $\concp$ is a fixed hyperparameter satisfying $\concp > 0$. 

The name ``stick-breaking'' originates from thinking of the unit interval as
a stick of length one. At each round $k$, only some of the stick remains;
$\beta_{k}$ describes the proportion of the remaining stick that is
broken off in round $k$, and $\prw_{DP,k}$ describes the total
amount of remaining stick that is broken off in round $k$. By construction,
not only is each $\prw_{DP,k} \in (0,1)$ but in fact the $\prw_{DP,k}$
add to one (the total stick length) and thus describe a distribution.

\eq{dp_stick} is called a \emph{size-biased} representation
for the following reason. Since the weights $\{\prw_{DP,k}\}_{k=1}^{\infty}$
describe a distribution, we can make draws from this distribution;
each such draw is sometimes thought of as a multinomial draw with
a single trial. In that vein,
typically
we imagine that our data points $\obsm_{mult,n}$ are described as iid
draws conditioned on $\prm_{DP}$, where $\obsm_{mult,n}$ is a random
measure with just a single atom:
\begin{align} \label{eq:dp_like}
\begin{split}
	\obsm_{mult,n} = \delta_{\prl_{mult,n}}; \quad
	\prl_{mult,n} = \prl_{k} \textrm{ with probability } \prw_{DP,k}.
\end{split}
\end{align}
Then the limiting proportion of data points $\obsm_{mult,n}$
with an atom at $\prl_{mult,1}$ (the first atom location chosen)
is $\prw_{DP,1}$. The limiting proportion of
data points with an atom at the next unique atom location chosen
will have size $\prw_{DP,2}$, and so on \citep{broderick:2013:cluster}.

The representation in \eq{dp_stick} is so useful because there is
a familiar, finite-dimensional distribution for each of the atom
weights $\prw_{DP,k}$ of the 
random measure $\prm_{DP}$. This representation allows
approximate
inference via truncation \citep{ishwaran:2001:gibbs} or exact inference via
slice sampling \citep{walker:2007:sampling, kalli:2011:slice}. 

Since the weights $\{\prw_{DP,k}\}_{k=1}^{\infty}$ are constrained to sum 
to one, the Dirichlet process is
not a CRM.\footnote{In fact, the Dirichlet process 
is a normalized gamma process (cf.~\ex{plp_gp_auto_conj})
\citep{ferguson:1973:bayesian}.}
Indeed, there has been much work on size-biased representations for more general
normalized random measures, which include the Dirichlet process as just one example
\citep{perman:1992:size,pitman:1996:random,pitman:1996:some,pitman:2003:poisson}.

By contrast, we here wish to explore size-biasing for non-normalized CRMs. 
In the normalized CRM case, we considered which atom of a random discrete probability measure
was drawn first and what is the distribution of that atom's size.
In the non-normalized CRM case considered in the present work,
when drawing $\obsm$ conditional on $\prm$,
there may be multiple atoms (or one atom or no atoms) of $\prm$ that correspond
to non-zero atoms in $\obsm$. The number will always be finite though by
\asu{fin}.
In this non-normalized CRM case, we wish to consider the sizes of all such atoms
in $\prm$.
Size-biased representations have been developed in the past for particular CRM
examples, notably the beta process
\citep{paisley:2010:stick, broderick:2012:beta}.
And even though
there is typically no interpretation of these representations in terms of a
single stick representing a unit probability mass, they are sometimes
referred to as stick-breaking representations as a nod to the popularity of 
Dirichlet process stick-breaking.

In the beta process case, such size-biased representations
have already been shown to allow approximate inference
via truncation \citep{doshi:2009:variational, paisley:2011:variational}
or exact inference via slice sampling
\citep{teh:2007:stick,broderick:2015:combinatorial}.
Here we
provide general recipes for the creation 
of these representations and illustrate
our recipes by discovering previously unknown 
size-biased representations.

We have seen that a general CRM $\prm$ takes the form of an
a.s.\ discrete random measure:
\begin{equation} \label{eq:crm_discrete}
	\sum_{k=1}^{\infty} \prw_{k} \delta_{\prl_{k}}.
\end{equation}
The fixed-location atoms are straightforward to simulate; there are
finitely many by \asu{fix}, their locations
are fixed, and their weights are assumed to come from finite-dimensional
distributions. The infinite-dimensionality of the Bayesian nonparametric
CRM comes from the ordinary component (cf.\ \mysec{bnp} and \asu{inf}).
So far the only description we have of the ordinary component
is its generation from the countable infinity of points in a Poisson point process.
The next result constructively demonstrates that we
can represent the distributions of the CRM weights $\{\prw_{k}\}_{k=1}^{\infty}$
in \eq{crm_discrete} as a sequence of finite-dimensional distributions,
much as in the familiar Dirichlet process case.

\begin{theorem}[Size-biased representations] \label{thm:size_bias}
	Let $\prm$ be a completely random measure that satisfies \asus{fix} and
	\asuss{inf}; that is, $\prm$ is a CRM with $K_{fix}$ fixed atoms such that
	$K_{fix} < \infty$ and such that the $k$th atom can be written
	$\prw_{fix,k} \delta_{\prl_{fix,k}}$. The ordinary component of $\prm$ has rate measure
	$$
		\fullratem(d\prw \times d\prl) = \wratem(d\prw) \cdot \dordloc(d\prl),
	$$
	where $\dordloc$ is a proper distribution and $\wratem(\obssp) = \infty$.
	Write $\prm = \sum_{k=1}^{\infty} \prw_{k} \delta_{\prl_{k}}$, and let
	$\obsm_{n}$ be generated iid given $\prm$ according to 
	$\obsm_{n} = \sum_{k=1}^{\infty} \obsw_{n,k} \delta_{\prl_{k}}$ with
	$\obsw_{n,k} \indep \denslike(\obsw | \prw_{k})$ for proper, discrete
	probability mass function $\denslike$. And suppose $\obsm_{n}$ and $\prm$ jointly satisfy
	\asu{fin} so that
	$$
		\sum_{\obsw=1}^{\infty} \int_{\prw \in \obssp} \wratem(d\prw) \denslike(\obsw | \prw) < \infty.
	$$
	
	Then we can write
	\begin{align} \label{eq:size_bias_gen}
	\begin{split}
		\prm &= \sum_{m=1}^{\infty} \sum_{\obsw = 1}^{\infty}
			\sum_{j=1}^{\rho_{m,\obsw}} \prw_{m,\obsw,j} \delta_{\prl_{m,\obsw,j}} \\
		\prl_{m,\obsw,k} &\iid \dordloc \textrm{ iid across $m, \obsw, j$} \\
		\rho_{m,\obsw} &\indep \pois\left(
			\rho \left\vert
				\int_{\prw} \wratem(d\prw)
				\denslike(0 | \prw)^{m-1} \denslike(\obsw | \prw)
			\right. \right)
			\textrm{ across $m, \obsw$} \\
		\prw_{m,\obsw,j} &\indep F_{size,m,\obsw}(d\prw) \propto \wratem(d\prw) \denslike(0 | \prw)^{m-1} \denslike(\obsw | \prw) \\
			& \textrm{ iid across $j$ and independently across $m, \obsw$}.
	\end{split}
	\end{align}
\end{theorem}

%
\begin{proof}
	By construction, $\prm$ is an a.s.\ discrete random measure
	with a countable infinity of atoms. 
	Without loss of generality, suppose that for every (non-zero) value of
	an atom weight
	$\prw$, there is a non-zero probability of generating an atom with
	non-zero weight $\obsw$ in the likelihood. Now suppose we
	generate $\obsm_{1}, \obsm_{2}, \ldots$. Then, for every atom
	$\prw \delta_{\prl}$
	of $\prm$, there exists some finite $n$ with an atom at $\prl$.
	Therefore, we can enumerate all of the atoms of $\prm$ by enumerating
	\begin{itemize}
		\item Each atom $\prw \delta_{\prl}$ such that there is an atom in
	$\obsm_{1}$ at $\prl$.
		\item Each atom $\prw \delta_{\prl}$ such that there is an atom in $\obsm_{2}$
		at $\prl$ but there is not an atom in $\obsm_{1}$ at $\prl$.
		\\ $\vdots$
		\item Each atom $\prw \delta_{\prl}$ such that there is an atom
		in $\obsm_{m}$ at $\prl$ but there is not an atom in any of $\obsm_{1}, \obsm_{2}, \ldots, \obsm_{m-1}$
		at $\prl$.
		\\ $\vdots$
	\end{itemize}
	Moreover, on the $m$th round of this enumeration, we can further break down the enumeration
	by the value of the observation $\obsm_{m}$ at the atom location:
	\begin{itemize}
		\item Each atom $\prw \delta_{\prl}$ such that there is an atom
		in $\obsm_{m}$ \textbf{of weight $1$} at $\prl$ but there is not an atom in any of $\obsm_{1}, \obsm_{2}, \ldots, \obsm_{m-1}$
		at $\prl$.
		\item Each atom $\prw \delta_{\prl}$ such that there is an atom
		in $\obsm_{m}$ \textbf{of weight $2$} at $\prl$ but there is not an atom in any of $\obsm_{1}, \obsm_{2}, \ldots, \obsm_{m-1}$
		at $\prl$.
		\\ $\vdots$
		\item Each atom $\prw \delta_{\prl}$ such that there is an atom
		in $\obsm_{m}$ \textbf{of weight $\obsw$} at $\prl$ but there is not an atom in any of $\obsm_{1}, \obsm_{2}, \ldots, \obsm_{m-1}$
		at $\prl$.
		\\ $\vdots$
	\end{itemize}
	
	Recall that the values $\prw_{k}$ that form the weights of $\prm$ are generated according
	to a Poisson point process with rate measure $\wratem(d\prw)$. So, on the first round, the 
	values of $\prw_{k}$ such that $\obsw_{1,k} = \obsw$ also holds are generated according
	to a thinned Poisson point process with rate measure
	$$
		\wratem(d\prw) \denslike(\obsw | \prw).
	$$
	In particular, since the rate measure has finite total mass by \asu{fin}, 
	we can define
	$$
		M_{1,\obsw} := \int_{\prw} \wratem(d\prw) \denslike(\obsw | \prw),
	$$
	which will be finite.
	Then the number of atoms $\prw_{k}$ for which $\obsw_{1,k} = \obsw$ is
	$$
		\rho_{1,\obsw} \sim \pois(\rho | M_{1,\obsw}).
	$$
	And each such $\prw_{k}$ has weight with distribution
	$$
		F_{size,1,\obsw}(d\prw) \propto \wratem(d\prw) \denslike(\obsw | \prw).
	$$
	Finally, note from \thm{post} that the posterior
	$\prm | \obsm_{1}$ has weight rate measure
	$$
		\wratem_{1}(d\prw) := \wratem(d\prw) \denslike(0 | \prw).
	$$
	
	Now take any $m > 1$. Suppose, inductively, that the ordinary component of the posterior 
	$\prm | \obsm_{1}, \ldots, \obsm_{m-1}$ has weight rate measure
	$$
		\wratem_{m-1}(d\prw) := \wratem(d\prw) \denslike(0 | \prw)^{m-1}.
	$$
	The atoms in this ordinary component have been selected precisely because
	they have not appeared in any of $\obsm_{1}, \ldots, \obsm_{m-1}$.
	As for $m=1$, we have that the
	atoms $\prw_{k}$ in this ordinary component with corresponding
	weight in $\obsm_{m}$ equal to $\obsw$ are formed by a thinned Poisson point
	process, with rate measure
	$$
		\wratem_{m-1}(d\prw) \denslike(\obsw | \prw) = \wratem(d\prw) \denslike(0 | \prw)^{m-1} \denslike(\obsw | \prw).
	$$
	Since the rate measure has finite total mass by \asu{fin}, 
	we can define
	$$
		M_{m,\obsw} := \int_{\prw} \wratem(d\prw) \denslike(0 | \prw)^{m-1} \denslike(\obsw | \prw),
	$$
	which will be finite.
	Then the number of atoms $\prw_{k}$ for which $\obsw_{1,k} = \obsw$ is
	$$
		\rho_{m,\obsw} \sim \pois(\rho | M_{m,\obsw}).
	$$
	And each such $\prw_{k}$ has weight
	$$
		F_{size,m,\obsw} \propto \wratem(d\prw) \denslike(0 | \prw)^{m-1} \denslike(\obsw | \prw).
	$$
	Finally, note from \thm{post} that the posterior
	$\prm | \obsm_{1:m}$, which can be thought of as generated by prior $\prm | \obsm_{1:(m-1)}$
	and likelihood $\obsm_{m} | \prm$, has weight rate measure
	$$
		\wratem(d\prw) \denslike(0 | \prw)^{m-1} \denslike(0 | \prw) = \wratem_{m}(d\prw),
	$$	
	confirming the inductive hypothesis.
	
	Recall that every atom of $\prm$ is found in exactly one of these rounds
	and that $\obsw \in \Zplus$. Also recall that the atom locations may
	be generated independently and identically across atoms, and independently from
	all the weights, according to proper distribution $\dordloc$ (\mysec{prior_like}).
	To summarize, we have then
	$$
		\prm = \sum_{m=1}^{\infty} \sum_{\obsw = 1}^{\infty}
			\sum_{j=1}^{\rho_{m,\obsw}} \prw_{m,\obsw,j} \delta_{\prl_{m,\obsw,j}},
	$$
	where
	\begin{align*}
		\prl_{m,\obsw,k} &\iid \dordloc \textrm{ iid across $m, \obsw, j$} \\
		M_{m,\obsw} &= \int_{\prw} \wratem(d\prw) \denslike(0 | \prw)^{m-1} \denslike(\obsw | \prw) 
			\textrm{ across $m,\obsw$} \\
		\rho_{m,\obsw} &\indep \pois(\rho | M_{m,\obsw}) \textrm{ across $m, \obsw$} \\
		F_{size,m,\obsw}(d\prw) &\propto \wratem(d\prw) \denslike(0 | \prw)^{m-1} \denslike(\obsw | \prw)
			\textrm{ across $m, \obsw$} \\
		\prw_{m,\obsw,j} &\indep F_{size,m,\obsw}(d\prw)
			\textrm{ iid across $j$ and independently across $m, \obsw$},
	\end{align*}
	as was to be shown.
\end{proof}

The following corollary gives a more detailed recipe for the
calculations in \thm{size_bias}
when the prior is in a conjugate exponential CRM to the likelihood.

\begin{corollary}[Exponential CRM size-biased representations] \label{cor:size}
	Let $\prm$ be an exponential CRM with no fixed-location atoms (thereby trivially 
	satisfying \asu{fix}) 
	such that \asu{inf} holds.
	
	Let $\obsm$ be generated conditional on $\prm$
	according to an exponential CRM with 
	fixed-location atoms at $\{\prl_{k}\}_{k=1}^{\infty}$
	and no ordinary component. 
	Let the distribution of the weight $\obsw_{n,k}$
	at $\prl_{k}$
	have probability mass function
	$$
		\denslike(\obsw | \prw_{k})
			= \basem(\obsw) \exp\left\{ \langle \natpar(\prw_{k}), \suffstat(\obsw) \rangle 
				- \logpart(\prw_{k}) \right\}.
	$$
	
	Suppose that $\prm$ and $\obsm$ jointly
	satisfy \asu{fin}. And let $\prm$ be conjugate to $\obsm$ as in \thm{auto_conj}.
	Then we can write
	 \begin{align} \label{eq:size_bias_exp_fam}
	 \begin{split}
		\prm &= \sum_{m=1}^{\infty}
				\sum_{\obsw=1}^{\infty}
				\sum_{j=1}^{\rho_{m, \obsw}}
					\prw_{m, \obsw, j} \delta_{\prl_{m, \obsw, j}} \\
		\prl_{m, \obsw, j}
			&\iid \dordloc \quad \textrm{iid across $m, \obsw, j$} \\
		M_{m,\obsw}
			&= \massp \cdot \basem(0)^{m-1} \basem(\obsw)
				\cdot \exp\left\{ B(\xi + (m-1) \suffstat(0) + \suffstat(\obsw), \lambda + m) \right\} \\
		\rho_{m, \obsw}
			&\indep \pois\left( \rho | M_{m,\obsw} \right) \\
			& \textrm{independently across $m, \obsw$} \\
		\prw_{m, \obsw, j}
			&\indep f_{size, m, \obsw}(\prw) \; d\prw \\
			&= \exp\left\{ \langle \xi + (m-1) \suffstat(0) + \suffstat(\obsw), \natpar(\prw) \rangle
				+ (\lambda + m)[-\logpart(\prw)] \right. \\
			& \quad \left. {}
				- B(\xi + (m-1) \suffstat(0) + \suffstat(\obsw), \lambda + m) \right\} \\
			& \textrm{iid across $j$ and independently across $m, \obsw$}.
	\end{split}
	\end{align}
\end{corollary}

%
\begin{proof}
	The corollary follows from \thm{size_bias} by plugging in the particular 
	forms for $\wratem(d\prw)$ and $\denslike(\obsw | \prw)$.
	
	In particular,
	\begin{align*}
		M_{m, \obsw}
			&= \int_{\prw \in \obssp} \wratem(d\prw) \denslike(0 | \prw)^{m-1} \denslike(\obsw | \prw) \\
			&= \int_{\prw \in \obssp} 
				\massp \exp\left\{ \langle \xi, \natpar(\prw) \rangle
				+ \lambda \left[ -\logpart(\prw) \right] \right\} \\
			& \quad {}
				\cdot \left[
				\basem(0) \exp\left\{ \langle \natpar(\prw), \suffstat(0) \rangle 
				- \logpart(\prw) \right\} \right]^{m-1} \\
			& \quad {}
				\cdot 
				\basem(\obsw) \exp\left\{ \langle \natpar(\prw), \suffstat(\obsw) \rangle 
				- \logpart(\prw) \right\} \;
				d\prw \\
			&= \massp \basem(0)^{m-1} \basem(\obsw)
				\exp\left\{
					\priorlogpart\left(\xi + (m-1)\suffstat(0) + \suffstat(\obsw), \lambda + m \right)
				\right\},
	\end{align*}
\end{proof}

\cor{size} can be used to
find
the known size-biased representation of the beta process \citep{thibaux:2007:hierarchical};
we demonstrate this derivation in detail in \ex{bp_size} in \app{size}. Here we
use \cor{size} to discover a new size-biased representation of the gamma process.

\begin{example}
	Let $\prm$ be a gamma process, and let $\obsm_{n}$ be iid Poisson
	likelihood processes conditioned on $\prm$ for each $n$
	as in \ex{plp_gp_auto_conj}. That is, we have
	$$
		\wratem(d\prw) = \massp \prw^{\xi} e^{-\lambda \prw} \; d\prw.
	$$
	And
	$$
		\denslike(\obsw | \prw_{k})
			= \frac{1}{\obsw !} \prw_{k}^{\obsw} e^{-\prw_{k}}
	$$
	with
	\begin{align*}
		\massp > 0, \quad
		\xi \in (-2,-1], \quad
		\lambda > 0; \quad
		\xi_{fix,k} > -1 \textrm{ and } \lambda_{fix,k} > 0 \quad \textrm{for all $k \in [K_{prior,fix}]$}
	\end{align*}
	by \ex{plp_gp_auto_conj}.
	
	We can pick out the following components of $\denslike$:
	\begin{align*}
		\basem(\obsw) = \frac{1}{\obsw !}, \quad
		\suffstat(\obsw) = \obsw, \quad
		\natpar(\prw) = \log(\prw), \quad
		\logpart(\prw) = \prw.
	\end{align*}
	Thus, by \cor{size},
	we have
	\begin{align*}
		f_{size, m, \obsw}(\prw)
			\propto \prw^{\xi + \obsw} e^{-(\lambda + m) \prw}
			\propto \ga\left( \prw \left\vert \xi + \obsw + 1, \lambda + m \right. \right).
	\end{align*}
	We summarize the representation that follows from \cor{size} in the following
	result.
	
	\begin{corollary}
	Let the gamma process be a CRM $\prm$ with fixed-location atom weight distributions
	as in \eq{gp_fixed_atom} and ordinary component weight measure as in
	\eq{gp_ord}. Then we may write
	 \begin{align*}
		\prm &= \sum_{m=1}^{\infty}
				\sum_{\obsw=1}^{\infty}
				\sum_{j=1}^{\rho_{m, \obsw}}
					\prw_{m, \obsw, j} \delta_{\prl_{m, \obsw, j}} \\
		\prl_{m, \obsw, j}
			&\iid \dordloc \quad \textrm{ iid across $m, \obsw, j$} \\
		M_{m,\obsw}
			&=  \gamma \cdot \frac{1}{\obsw!}
				\cdot \Gamma(\xi + \obsw + 1) \cdot (\lambda + m)^{-(\xi + \obsw + 1)}
				\textrm{ across $m, \obsw$} \\
		\rho_{m, \obsw}
			&\indep \pois\left( \rho | M_{m,\obsw} \right)
				\textrm{ across $m, \obsw$} \\
		\prw_{m, \obsw, j}
			&\indep \ga\left( \prw \left\vert \xi + \obsw + 1, \lambda + m \right. \right) \\
			& \textrm{ iid across $j$ and independently across $m, \obsw$}.
	\end{align*}
	\end{corollary}
\end{example}


\section{Marginal processes} \label{sec:marg}

In \mysec{size}, although we conceptually made use 
of the observations $\{\obsm_{1}, \obsm_{2}, \ldots\}$,
we focused on a representation of
the prior $\prm$: cf.\ \eqs{size_bias_gen} and
\eqss{size_bias_exp_fam}.
In this section, we provide a representation of the
marginal
of $\obsm_{1:N}$, with $\prm$ integrated out.

The canonical example of a marginal process again comes
from the Dirichlet process (DP). In this case, the full model consists
of the DP-distributed prior on $\prm_{DP}$ (as in \eq{dp_stick})
together with the likelihood for $\obsm_{mult,n}$ conditional on
$\prm_{DP}$ (iid across $n$) described by \eq{dp_like}. Then
the marginal distribution of $\obsm_{mult,1:N}$ is described by
the \emph{Chinese restaurant process}. This marginal takes the following form.

For each $n = 1, 2, \ldots, N$,
\begin{enumerate}
	\item Let $\{\prl_{k}\}_{k=1}^{K_{n-1}}$ be the union of atom locations in
	$\obsm_{mult,1}, \ldots, \obsm_{mult,n-1}$.
	Then
	$\obsm_{mult,n} | \obsm_{mult,1}, \ldots, \obsm_{mult,n-1}$ 
	has a single atom at $\prl$, where
	\begin{align*}
		\prl &= \left\{ \begin{array}{ll}
				\prl_{k} & \textrm{ with probability $\propto$ }
						\sum_{k=1}^{K_{n-1}} \obsm_{mult,m}(\{\prl_{k}\}) \\
				\prl_{new} & \textrm{ with probability $\propto$ } \concp
			\end{array} \right. \\
		\prl_{new} &\sim \dordloc
	\end{align*}
\end{enumerate}

In the case of CRMs, the canonical example of a marginal
process is the Indian buffet process \citep{griffiths:2006:infinite}.
Both the Chinese restaurant
process and Indian buffet process have proven popular for inference since
the underlying infinite-dimensional prior is integrated out in these processes
and only the finite-dimensional marginal remains. By \asu{fin},
we know that the marginal will generally
be finite-dimensional for our CRM Bayesian models.
And thus we have the following general marginal representations 
for such models.

\begin{theorem}[Marginal representations] \label{thm:marg}
	Let $\prm$ be a completely random measure that satisfies \asus{fix} and
	\asuss{inf}; that is, $\prm$ is a CRM with $K_{fix}$ fixed atoms such that
	$K_{fix} < \infty$ and such that the $k$th atom can be written
	$\prw_{fix,k} \delta_{\prl_{fix,k}}$. The ordinary component of $\prm$ has rate measure
	$$
		\fullratem(d\prw \times d\prl) = \wratem(d\prw) \cdot \dordloc(d\prl),
	$$
	where $\dordloc$ is a proper distribution and $\wratem(\obssp) = \infty$.
	Write $\prm = \sum_{k=1}^{\infty} \prw_{k} \delta_{\prl_{k}}$, and let
	$\obsm_{n}$ be generated iid given $\prm$ according to 
	$\obsm_{n} = \sum_{k=1}^{\infty} \obsw_{n,k} \delta_{\prl_{k}}$ with
	$\obsw_{n,k} \indep \denslike(\obsw | \prw_{k})$ for proper, discrete
	probability mass function $\denslike$. And suppose $\obsm_{n}$ and $\prm$ jointly satisfy
	\asu{fin} so that
	$$
		\sum_{\obsw=1}^{\infty} \int_{\prw \in \obssp} \wratem(d\prw) \denslike(\obsw | \prw) < \infty.
	$$
	
	Then the marginal distribution of $\obsm_{1:N}$ is the same as that provided by the 
	following construction.
	
	For each $n = 1, 2, \ldots, N$,
	\begin{enumerate}
		\item Let $\{\prl_{k}\}_{k=1}^{K_{n-1}}$ be the union of atom locations in
		$\obsm_{1}, \ldots, \obsm_{n-1}$. Let $\obsw_{m,k} := \obsm_{m}(\{\prl_{k}\})$.
		Let $\obsw_{n,k}$ denote the weight of
		$\obsm_{n} | \obsm_{1}, \ldots, \obsm_{n-1}$ at $\prl_{k}$. Then
		$\obsw_{n,k}$ has distribution described by the following probability
		mass function:
		$$
			\denscond\left( \obsw_{n,k} = \obsw \left\vert \obsw_{1:(n-1),k} \right. \right)
				= \frac{
					\int_{\prw \in \obssp} 
					\wratem(d\prw) \denslike(\obsw | \prw) \prod_{m=1}^{n-1} \denslike(\obsw_{m,k} | \prw)
				}{
					\int_{\prw \in \obssp} 
					\wratem(d\prw) \prod_{m=1}^{n-1} \denslike(\obsw_{m,k} | \prw)
				}.
		$$
		\item For each $\obsw = 1, 2, \ldots$
		\begin{itemize}
			\item $\obsm_{n}$ has $\rho_{n,\obsw}$ new atoms. That is,
			$\obsm_{n}$ has atoms at locations $\{\prl_{n,\obsw,j}\}_{j=1}^{\rho_{n,\obsw}}$,
			where
			$$
				\{\prl_{n,\obsw,j}\}_{j=1}^{\rho_{n,\obsw}}
					\cap
					\{\prl_{k}\}_{k=1}^{K_{n-1}}
					= \emptyset
					\quad \textrm{ a.s.}
			$$
			Moreover,
			\begin{align*}
				\rho_{n,\obsw}
					&\indep \pois\left( \rho
						\left\vert
							\int_{\prw} \wratem(d\prw)
							\denslike(0 | \prw)^{n-1} \denslike(\obsw | \prw)
						\right. \right)
					\textrm{ across $n, \obsw$} \\ 
				\prl_{n,\obsw,j} &\iid \dordloc(d\prl) 
					\textrm{ across $n, \obsw, j$}.
			\end{align*}
		\end{itemize}
	\end{enumerate}
\end{theorem}

%
\begin{proof}
	We saw in the proof of \thm{size_bias} that the marginal for $\obsm_{1}$
	can be expressed as follows.
	For each $\obsw \in \Zplus$, there are $\rho_{1,\obsw}$
	atoms of $\obsm_{1}$ with weight $\obsw$, where
	\begin{align*}
		\rho_{1,\obsw} &\indep \pois\left(\int_{\prw} \wratem(d\prw) \denslike(\obsw | \prw) \right)
			\textrm{ across $\obsw$}.
	\end{align*}
	These atoms have locations
	$\{\prl_{1,\obsw,j}\}_{j=1}^{\rho_{1,\obsw}}$, where
	\begin{align*}
		\prl_{1,\obsw,j} &\iid \dordloc(d\prl)
			\textrm{ across $\obsw, j$}.
	\end{align*}
	For the upcoming induction, let $K_{1} := \sum_{\obsw = 1}^{\infty} \rho_{1,\obsw}$.
	And let $\{\prl_{k}\}_{k=1}^{K_{1}}$ be the (a.s.\ disjoint by assumption)
	union of the sets $\{\prl_{1,\obsw,j}\}_{j=1}^{\rho_{1,\obsw}}$ across $\obsw$.
	Note that $K_{1}$ is finite by \asu{fin}.
	
	We will also find it useful in the upcoming induction to let $\prm_{post,1}$
	have the distribution of $\prm | \obsm_{1}$.
	Let $\prw_{post,1,\obsw,j} = \prm_{post,1}(\{\prl_{1,\obsw,j}\})$.
	By \thm{post} or the proof of \thm{size_bias}, we have that
	\begin{align*}
		\prw_{post,1,\obsw,j}
			&\indep F_{post,1,\obsw,j}(d\prw)
			\propto \wratem(d\prw) \denslike(\obsw | \prw) \\
			& \quad
				\textrm{ independently across $\obsw$ and iid across $j$}.
	\end{align*}

	Now take any $n > 1$. Inductively,
	we assume $\{\prl_{n-1,k}\}_{k=1}^{K_{n-1}}$ is the union of
	all the atom locations of $\obsm_{1}, \ldots, \obsm_{n-1}$.
	Further assume $K_{n-1}$ is finite.
	Let $\prm_{post, n-1}$ have the distribution
	of $\prm | \obsm_{1}, \ldots, \obsm_{n-1}$. Let 
	$\prw_{n-1,k}$ be the weight of $\prm_{post,n-1}$
	at $\prl_{n-1,k}$. And, for any $m \in [n-1]$,
	let $\obsw_{m,k}$ be the weight of 
	$\obsm_{m}$ at $\prl_{n-1,k}$.
	We inductively assume that 
	\begin{align} \label{eq:marg_step_post}
	\begin{split}
		\prw_{n-1,k}
			&\indep F_{n-1,k}(d\prw)
			\propto \wratem(d\prw) \prod_{m=1}^{n-1} \denslike(\obsw_{m,k} | \prw) \\
			& \textrm{independently across $k$}.
	\end{split}
	\end{align}

	Now let $\prl_{n,k}$ equal $\prl_{n-1,k}$ for $k \in [K_{n-1}]$.
	Let $\obsw_{n,k}$ denote the weight of $\obsm_{n}$
	at $\prl_{n,k}$ for $k \in [K_{n-1}]$.
	Conditional on the atom weight of $\prm$ at $\prl_{n,k}$,
	the atom weights of $\obsm_{1}, \ldots, \obsm_{n-1}, \obsm_{n}$
	are independent. Since the atom weights of $\prm$ are
	independent as well, we have that
	$\obsw_{n,k} | \obsm_{1}, \ldots, \obsm_{n-1}$
	has the same distribution as $\obsw_{n,k} | \obsw_{1,k}, \ldots, \obsw_{n-1,k}$.
	We can write the probability mass function of this distribution as follows.
	\begin{align*}	
		\lefteqn{ \denscond\left( \obsw_{n,k} = \obsw
				\left\vert
					\obsw_{1,k}, \ldots, \obsw_{n-1,k} 
				\right. \right) } \\
			&= \int_{\prw \in \obssp} 
				F_{n-1,k}(d\prw) \denslike(\obsw | \prw) \\
			&= \frac{
					\int_{\prw \in \obssp} 
					\left[\wratem(d\prw) \prod_{m=1}^{n-1} \denslike(\obsw_{m,k} | \prw) \right] \cdot \denslike(\obsw | \prw)
				}{
					\int_{\prw \in \obssp} 
					\wratem(d\prw) \prod_{m=1}^{n-1} \denslike(\obsw_{m,k} | \prw)
				},
	\end{align*}
	where the last line follows from \eq{marg_step_post}.
	
	We next show the inductive hypothesis in \eq{marg_step_post} holds for $n$
	and $k \in [K_{n-1}]$.
	Let $\obsw_{n,k}$ denote the weight
	of $\obsm_{n}$ at $\prl_{n,k}$ for $k \in [K_{n-1}]$.
	Let
	$F_{n,k}(d\prw)$ denote the distribution of $\obsw_{n,k}$ and note that 
	\begin{align*}
		F_{n,k}(d\prw)
			&\propto
			F_{n-1,k}(d\prw) \cdot \denslike(\obsw_{n,k} | \prw) \\
			&= \wratem(d\prw) \prod_{m=1}^{n} \denslike(\obsw_{m,k} | \prw),
	\end{align*}
	which agrees with \eq{marg_step_post} for $n$ when we assume the result for $n-1$.
	
	The previous development covers atoms that are present in at least one of
	$\obsm_{1}, \ldots, \obsm_{n-1}$. Next we consider new atoms
	in $\obsm_{n}$; that is, we consider atoms in $\obsm_{n}$ for which there are
	no atoms at the same location in any of $\obsm_{1}, \ldots, \obsm_{n-1}$.
	
	We saw in the proof of \thm{size_bias} that,
	for each $\obsw \in \Zplus$, there are $\rho_{n,\obsw}$
	new atoms of $\obsm_{n}$ with weight $\obsw$ such that
	\begin{align*}
		\rho_{n,\obsw}
			&\indep \pois\left( \rho 
				\left\vert
					\int_{\prw} \wratem(d\prw)
					\denslike(0 | \prw)^{n-1} \denslike(\obsw | \prw)
				\right. \right)
			\textrm{ across $\obsw$}.
	\end{align*}
	These new atoms have locations
	$\{\prl_{n,\obsw,j}\}_{j=1}^{\rho_{n,\obsw}}$ with
	\begin{align*}
		\prl_{n,\obsw,j} &\iid \dordloc(d\prl)
			\textrm{ across $\obsw, j$}.
	\end{align*}
	By \asu{fin}, $\sum_{\obsw=1}^{\infty} \rho_{n,\obsw} < \infty$. So
	$$
		K_{n} := K_{n-1} + \sum_{\obsw=1}^{\infty} \rho_{n,\obsw}
	$$
	remains finite.
	Let $\prl_{n,k}$ for $k \in \{K_{n-1} + 1, \ldots, K_{n}\}$ index these
	new locations.
	Let 
	$\prw_{n,k}$ be the weight of $\prm_{post,n}$
	at $\prl_{n,k}$ for $k \in \{K_{n-1} + 1, \ldots, K_{n}\}$.
	And let $\obsw_{n,k}$ be the value of $\obsm$ at $\prl_{n,k}$.
	
	We check that the inductive hypothesis holds. By repeated
	application of \thm{post},
	the ordinary component of $\prm | \obsm_{1}, \ldots, \obsm_{n-1}$
	has rate measure
	$$
		\wratem(d\prw) \denslike(0 | \prw)^{n-1}.
	$$
	So, again by \thm{post}, we have that
	\begin{align*}
		\prw_{n,k}
			&\indep F_{n.k}(d\prw)
			\propto \wratem(d\prw) \denslike(0 | \prw)^{n-1} \denslike(\obsw_{n,k} | \prw).
	\end{align*}
	Since $\obsm_{m}$ has value 0 at $\prl_{n,k}$ for $m \in \{1,\ldots,n-1\}$
	by construction, we have that the inductive hypothesis holds.
\end{proof}

As in the case of size-biased representations (\mysec{size} and \cor{size}),
we can find a more detailed recipe when the prior is in
a conjugate exponential CRM to the likelihood.

\begin{corollary}[Exponential CRM marginal representations] \label{cor:marg}
	Let $\prm$ be an exponential CRM with no fixed-location atoms (thereby trivially 
	satisfying \asu{fix}) 
	such that \asu{inf} holds.
	
	Let $\obsm$ be generated conditional on $\prm$
	according to an exponential CRM with 
	fixed-location atoms at $\{\prl_{k}\}_{k=1}^{\infty}$
	and no ordinary component. 
	Let the distribution of the weight $\obsw_{n,k}$
	at $\prl_{k}$
	have probability mass function
	$$
		\denslike(\obsw | \prw_{k})
			= \basem(\obsw) \exp\left\{ \langle \natpar(\prw_{k}), \suffstat(\obsw) \rangle 
				- \logpart(\prw_{k}) \right\}.
	$$
	
	Suppose that $\prm$ and $\obsm$ jointly
	satisfy \asu{fin}. And let $\prm$ be conjugate to $\obsm$ as in \thm{auto_conj}.
	Then the marginal distribution of $\obsm_{1:N}$ is the same as that provided by the 
	following construction.
	
	For each $n = 1, 2, \ldots, N$,
	\begin{enumerate}
		\item Let $\{\prl_{k}\}_{k=1}^{K_{n-1}}$ be the union of atom locations in
		$\obsm_{1}, \ldots, \obsm_{n-1}$. Let $\obsw_{m,k} := \obsm_{m}(\{\prl_{k}\})$.
		Let $\obsw_{n,k}$ denote the weight of
		$\obsm_{n} | \obsm_{1}, \ldots, \obsm_{n-1}$ at $\prl_{k}$. Then
		$\obsw_{n,k}$ has distribution described by the following probability
		mass function:
		\begin{align*}
			\lefteqn{ \denscond\left( \obsw_{n,k} = \obsw
					\left\vert
						\obsw_{1:(n-1),k}
					\right. \right)
					} \\
				&= \basem(\obsw) \exp\left\{
					-\priorlogpart(\xi + \sum_{m=1}^{n-1} \obsw_{m}, \lambda + n-1)
					+ \priorlogpart(\xi + \sum_{m=1}^{n-1} \obsw_{m} + \obsw, \lambda + n)
					\right\}.
		\end{align*}
		\item For each $\obsw = 1, 2, \ldots$
		\begin{itemize}
			\item $\obsm_{n}$ has $\rho_{n,\obsw}$ new atoms. That is,
			$\obsm_{n}$ has atoms at locations $\{\prl_{n,\obsw,j}\}_{j=1}^{\rho_{n,\obsw}}$,
			where
			$$
				\{\prl_{n,\obsw,j}\}_{j=1}^{\rho_{n,\obsw}}
					\cap
					\{\prl_{k}\}_{k=1}^{K_{n-1}}
					= \emptyset
					\quad \textrm{ a.s.}
			$$
			Moreover,
			\begin{align*}
				M_{n, \obsw} &:= \massp \cdot \basem(0)^{n-1} \basem(\obsw) \cdot \exp\left\{ \priorlogpart(\xi + (n-1) \suffstat(0) + \suffstat(\obsw), \lambda + n) \right\} \\
					& \textrm{ across $n, \obsw$} \\
				\rho_{n,\obsw} &\indep \pois\left( \rho \left\vert M_{n, \obsw} \right. \right)
					\textrm{ across $n, \obsw$} \\ 
				\prl_{n,\obsw,j} &\iid \dordloc(d\prl) 
					\textrm{ across $n, \obsw, j$}.
			\end{align*}
		\end{itemize}
	\end{enumerate}
\end{corollary}

%
\begin{proof}
		The corollary follows from \thm{marg} by plugging in the 
	forms for $\wratem(d\prw)$ and $\denslike(\obsw | \prw)$.
	
	In particular,
	\begin{align*}
		\lefteqn{
			\int_{\prw \in \obssp} 
				\wratem(d\prw) \prod_{m=1}^{n} \denslike(\obsw_{m,k} | \prw)
		} \\
			&= \int_{\prw \in \obssp} 
				\massp \exp\left\{ \langle \xi, \natpar(\prw) \rangle
				+ \lambda \left[ -\logpart(\prw) \right] \right\}
				\cdot \left[ \prod_{m=1}^{n}
				\basem(\obsw_{m,k}) \exp\left\{ \langle \natpar(\prw), \suffstat(\obsw_{m,k}) \rangle 
				- \logpart(\prw) \right\} \right] \\
			&= \massp \left[ \prod_{m=1}^{n} \basem(\obsw_{m,k}) \right] 
				\priorlogpart\left(\xi + \sum_{m=1}^{n} \suffstat(\obsw_{m,k}), \lambda + n \right).
	\end{align*}
	So
	\begin{align*}
		\lefteqn{ \denscond\left( \obsw_{n,k} = \obsw
				 \left\vert
				 	\obsw_{1:(n-1),k}
				\right. \right)
				} \\
			&= \frac{
				\int_{\prw \in \obssp} 
				\wratem(d\prw) \denslike(\obsw | \prw) \prod_{m=1}^{n-1} \denslike(\obsw_{m,k} | \prw)
			}{
				\int_{\prw \in \obssp} 
				\wratem(d\prw) \prod_{m=1}^{n-1} \denslike(\obsw_{m,k} | \prw)
			} \\
			&= \basem(\obsw) \exp\left\{ -\priorlogpart(\xi + \sum_{m=1}^{n-1} \obsw_{m}, \lambda + n-1) + \priorlogpart(\xi + \sum_{m=1}^{n-1} \obsw_{m} + \obsw, \lambda + n) \right\}.
	\end{align*}
\end{proof}

In \ex{bp_bep_marg} in \app{marg}
we show that \cor{marg} can be used to recover the Indian
buffet process marginal from a beta process prior together with 
a Bernoulli process likelihood. In the following example, 
we discover a new marginal for the Poisson likelihood process
with gamma process prior.

\begin{example}
	Let $\prm$ be a gamma process, and let $\obsm_{n}$ be iid Poisson
	likelihood processes conditioned on $\prm$ for each $n$
	as in \ex{plp_gp_auto_conj}. That is, we have
	$$
		\wratem(d\prw) = \massp \prw^{\xi} e^{-\lambda \prw} \; d\prw
	\quad 
	\textrm{ and }
	\quad
		\denslike(\obsw | \prw_{k})
			= \frac{1}{\obsw !} \prw_{k}^{\obsw} e^{-\prw_{k}}
	$$
	with
	\begin{align*}
		\massp > 0, \quad
		\xi \in (-2,-1], \quad
		\lambda > 0; \quad
		\xi_{fix,k} > -1 \textrm{ and } \lambda_{fix,k} > 0 \quad \textrm{for all $k \in [K_{prior,fix}]$}
	\end{align*}
	by \ex{plp_gp_auto_conj}.
	
	We can pick out the following components of $\denslike$:
	\begin{align*}
		\basem(\obsw) = \frac{1}{\obsw !}, \quad
		\suffstat(\obsw) = \obsw, \quad
		\natpar(\prw) = \log(\prw), \quad
		\logpart(\prw) = \prw.
	\end{align*}
	And we calculate
	\begin{align*}
		\exp\left\{ \priorlogpart(\xi, \lambda) \right\}
			= \int_{\prw \in \obssp}
				\exp\left\{ \langle \xi, \natpar(\prw) \rangle + \lambda [-\logpart(\prw)] \right\} \; d\prw
			= \int_{\prw \in \obssp}
				\prw^{\xi} e^{-\lambda \prw} 
			= \Gamma(\xi + 1) \lambda^{-(\xi + 1)}.
	\end{align*}
	
	So, for $k \in \Zstar$, we have
	\begin{align*}
		\mbp(\obsw_{n} = \obsw) 
			&= \basem(\obsw) \exp\left\{
				-\priorlogpart(\xi + \sum_{m=1}^{n-1} \obsw_{m}, \lambda + n-1)
				+ \priorlogpart(\xi + \sum_{m=1}^{n-1} \obsw_{m} + \obsw, \lambda + n)
				\right\} \\
			&= \frac{1}{\obsw!} \cdot
				\frac{
				(\lambda + n-1)^{\xi + \sum_{m=1}^{n-1} \obsw_{m} + 1}
				}{
				\Gamma(\xi + \sum_{m=1}^{n-1} \obsw_{m} + 1)
				} 
				\cdot
				\frac{
				\Gamma(\xi + \sum_{m=1}^{n-1} \obsw_{m} + \obsw + 1)
				}{
				(\lambda + n)^{\xi + \sum_{m=1}^{n-1} \obsw_{m} + \obsw + 1}
				} \\
			&= \frac{
				\Gamma(\xi + \sum_{m=1}^{n-1} \obsw_{m} + \obsw + 1)
				}{
				\Gamma(\obsw + 1)
				\Gamma(\xi + \sum_{m=1}^{n-1} \obsw_{m} + 1)
				} 
				\cdot
				\left(
					\frac{
						\lambda + n-1
					}{
						\lambda + n
					}
				\right)^{\xi + \sum_{m=1}^{n} \obsw_{m} + 1}
				\left(
					\frac{1}{
						\lambda + n
					}
				\right)^{\obsw} \\
			&= \negbin\left( \obsw 
				\left\vert
					\xi + \sum_{m=1}^{n-1} \obsw_{m} + 1, (\lambda + n)^{-1}
				\right. \right).
	\end{align*}
	And
	\begin{align*}
		M_{n, \obsw}
			&:= \massp \cdot \basem(0)^{n-1} \basem(\obsw)
				\cdot \exp\left\{ B(\xi + (n-1) \suffstat(0) + \suffstat(\obsw), \lambda + n) \right\} \\
			&= \massp \cdot \frac{1}{\obsw!}
				\cdot \Gamma(\xi + \obsw + 1) (\lambda+n)^{-(\xi + \obsw + 1)}.
	\end{align*}

	We summarize the marginal distribution representation of $\obsm_{1:N}$ that
	follows from \cor{marg} in the following result.
	
	\begin{corollary}
	Let $\prm$ be a gamma process with fixed-location atom weight distributions
	as in \eq{gp_fixed_atom} and ordinary component weight measure as in
	\eq{gp_ord}. Let $X_{n}$ be drawn, iid across $n$, conditional on $\prm$
	according to a Poisson likelihood process with fixed-location
	atom weight distributions as in \eq{plp_dens}.
	Then $X_{1:N}$ has the same distribution as the following construction.
	
	For each $n = 1, 2, \ldots, N$,
	\begin{enumerate}
		\item Let $\{\prl_{k}\}_{k=1}^{K_{n-1}}$ be the union of atom locations in
		$\obsm_{1}, \ldots, \obsm_{n-1}$. Let $\obsw_{m,k} := \obsm_{m}(\{\prl_{k}\})$.
		Let $\obsw_{n,k}$ denote the weight of
		$\obsm_{n} | \obsm_{1}, \ldots, \obsm_{n-1}$ at $\prl_{k}$. Then
		$\obsw_{n,k}$ has distribution described by the following probability
		mass function:
		\begin{align*}
			\denscond\left(
						\obsw_{n,k} = \obsw
						\left\vert
							\obsw_{1:(n-1),k}
						\right.
					\right)
				= \negbin\left( \obsw
					\left\vert
						\xi + \sum_{m=1}^{n-1} \obsw_{m,k} + 1, (\lambda + n)^{-1}
					\right. \right).
		\end{align*}
		\item For each $\obsw = 1, 2, \ldots$
		\begin{itemize}
			\item $\obsm_{n}$ has $\rho_{n,\obsw}$ new atoms. That is,
			$\obsm_{n}$ has atoms at locations $\{\prl_{n,\obsw,j}\}_{j=1}^{\rho_{n,\obsw}}$,
			where
			$$
				\{\prl_{n,\obsw,j}\}_{j=1}^{\rho_{n,\obsw}}
					\cap
					\{\prl_{k}\}_{k=1}^{K_{n-1}}
					= \emptyset
					\quad \textrm{ a.s.}
			$$
			Moreover,
			\begin{align*}
				M_{n, \obsw} &:= \massp \cdot \frac{1}{\obsw!}
					\cdot \frac{
							\Gamma(\xi + \obsw + 1)
						}{
							(\lambda+n)^{\xi + \obsw + 1}
						} \\
					& \textrm{ across $n, \obsw$} \\
				\rho_{n,\obsw} &\indep \pois\left( \rho \left\vert M_{n, \obsw} \right. \right)
					\textrm{ independently across $n, \obsw$} \\
				\prl_{n,\obsw,j} &\iid \dordloc(d\prl)
					\textrm{ independently across $n, \obsw$ and iid across $j$}.
			\end{align*}
		\end{itemize}
	\end{enumerate}
	\end{corollary}
	
\end{example}

\section{Discussion}

In the preceding sections, we have shown how to calculate posteriors for general 
CRM-based priors and likelihoods for Bayesian nonparametric models. We have also 
shown how to represent Bayesian nonparametric priors as a sequence of finite 
draws, and full Bayesian nonparametric models via finite marginals.  
We have introduced a notion of exponential families for CRMs, which we call 
exponential CRMs, that has allowed us to specify automatic Bayesian nonparametric
conjugate priors for exponential CRM likelihoods. And we have demonstrated
that our exponential CRMs allow particularly straightforward recipes
for size-biased and marginal representations of Bayesian nonparametric models.
Along the way, we have proved that the gamma process is a conjugate prior 
for the Poisson likelihood process and the beta prime process is a conjugate prior
for the odds Bernoulli process. We have discovered a size-biased representation
of the gamma process and a marginal representation of the gamma process 
coupled with a Poisson likelihood process.

All of this work has relied heavily on the description of Bayesian nonparametric
models in terms of completely random measures. As such, we have worked 
very particularly with pairings of real values---the CRM atom weights,
which we have interpreted as
trait frequencies or rates---together with trait descriptors---the CRM atom locations.
However, all of our proofs broke into essentially two parts: the fixed-location atom
part and the ordinary component part. The fixed-location atom development
essentially translated into the usual finite version of Bayes Theorem
and could easily be extended to full Bayesian models where
the prior describes a random element that need not be real-valued.
Moreover, the ordinary component development relied entirely on its
generation as a Poisson point process over a product space. It seems
reasonable to expect that our development might carry through when
the first element in this tuple need not be real-valued. And thus we believe
our results are suggestive of broader results over more general spaces.

\section*{Acknowledgements}

Support for this project was provided by
ONR under the Multidisciplinary University Research Initiative (MURI) program (N00014-11-1-0688).
T.~Broderick was supported by a Berkeley Fellowship. A.~C.~Wilson was supported
by an NSF Graduate Research Fellowship.

\newpage
\appendix

\section{Further automatic conjugate priors} \label{app:auto_conj}

We use \thm{auto_conj} to calculate automatic conjugate priors
for further exponential CRMs.

\begin{example} \label{ex:bep_auto_conj}
	Let $\obsm$ be generated according to a Bernoulli process
	as in \ex{bp_bep}. That is, $\obsm$ has an exponential CRM distribution
	with $K_{like,fix}$ fixed-location atoms, where $K_{like,fix} < \infty$
	in accordance with \asu{fix}:
	$$
		\obsm = \sum_{k=1}^{K_{like,fix}} \obsw_{like,k} \delta_{\prl_{like,k}}.
	$$
	The weight of the $k$th atom, $\obsw_{like,k}$,
	has support on $\{0,1\}$ and has a Bernoulli density
	with parameter $\prw_{k} \in (0,1]$:
	\begin{align*}
		\denslike(\obsw | \prw_{k})
			&= \prw_{k}^{\obsw} (1-\prw_{k})^{1-\obsw} \\
			&= \exp\left\{ \obsw \log(\prw_{k} / (1-\prw_{k})) + \log(1-\prw_{k}) \right\}
			.
	\end{align*}
	The final line is rewritten to emphasize the exponential family form
	of this density, with
	\begin{align*}
		\basem(\obsw) &= 1 \\
		\suffstat(\obsw) &= \obsw \\
		\natpar(\prw) &= \log\left( \frac{ \prw }{ 1-\prw } \right) \\
		\logpart(\prw) &= -\log(1-\prw).
	\end{align*}
	
	Then, by \thm{auto_conj}, $\obsm$ has a Bayesian nonparametric
	conjugate prior for
	$$
		\prm := \sum_{k=1}^{K_{like,fix}} \prw_{k} \delta_{\prl_{k}}.
	$$
	This conjugate prior has two parts.
	
	First, $\prm$ has a set of $K_{prior,fix}$ fixed-location atoms
	at some subset of the $K_{like,fix}$ fixed locations of $\obsm$. The $k$th
	such atom has random weight $\prw_{fix,k}$ with density
	\begin{align*}
		f_{prior,fix,k}(\prw)
			&= \exp\left\{ \langle \xi_{fix,k}, \natpar(\prw) \rangle + \lambda_{fix,k} 
				\left[ -\logpart(\prw) \right] - B(\xi_{fix,k}, \lambda_{fix,k}) \right\} \\
			&= \prw^{\xi_{fix,k}} (1-\prw)^{\lambda_{fix,k}-\xi_{fix,k} }
				\exp\left\{- B(\xi_{fix,k}, \lambda_{fix,k}) \right\} \\
			&= \tb\left( \prw
				\left\vert
					\xi_{fix,k} + 1, \lambda_{fix,k}-\xi_{fix,k} + 1
				\right. \right),
	\end{align*}
	where $\tb(\prw | a, b)$ denotes the beta density with shape parameters
	$a > 0$ and $b > 0$. So we must have fixed hyperparameters
	$\xi_{fix,k} > -1$ and $\lambda_{fix,k} > \xi_{fix,k} - 1$.
	Further,
	$$
		\exp\left\{- B(\xi_{fix,k}, \lambda_{fix,k}) \right\}
			= \frac{
					\Gamma(\lambda_{fix,k} + 2)
				}{
					\Gamma(\xi_{fix,k} + 1)
					\Gamma(\lambda_{fix,k}-\xi_{fix,k} + 1)
				}
	$$
	to ensure normalization.
	
	Second, $\prm$ has an ordinary component characterized by any proper distribution
	$\dordloc$ and weight rate measure
	\begin{align*}
		\wratem(d\prw)
			&= \massp \exp\left\{ \langle \xi, \natpar(\prw) \rangle
				+ \lambda \left[ -\logpart(\prw) \right] \right\} \; d\prw \\
			&= \massp \prw^{\xi} (1-\prw)^{\lambda - \xi} \; d\prw.
	\end{align*}
	
	Finally, we need to choose the allowable hyperparameter ranges for $\massp$,
	$\xi$, and $\lambda$. $\massp > 0$ ensures $\wratem$ is a measure.
	By \asu{inf}, we must have $\wratem(\obssp) = \infty$, so $\wratem$ 
	must represent an improper beta distribution. As such, we require
	either $\xi + 1 \le 0$ or $\lambda - \xi \le 0$.
	By \asu{fin}, we must have
	\begin{align*}
		\lefteqn{ \sum_{\obsw=1}^{\infty} \int_{\prw \in \obssp}
			\wratem(d\prw) \cdot \denslike(\obsw | \prw) } \\
			&= \int_{\prw \in (0,1]} \wratem(d\prw) \denslike(1 | \prw) \\
			& \textrm{ since the support of $\obsw$ is $\{0,1\}$ and the support of $\prw$ is $(0,1]$} \\
			&= \massp \int_{\prw \in (0,1]}
				\prw^{\xi} (1-\prw)^{\lambda - \xi} \; d\prw
				\cdot \prw \\
			&< \infty
	\end{align*}
	Since the integrand is the kernel of a beta distribution, the integral
	is finite if and only if $\xi + 2 > 0$ and $\lambda - \xi + 1 > 0$.
	
	Finally, then the hyperparameter restrictions can be summarized as:
	\begin{align*}
		\massp &> 0 \\
		\xi &\in (-2,-1] \\
		\lambda &> \xi - 1 \\
		\xi_{fix,k} &> -1 \textrm{ and } \lambda_{fix,k} > \xi_{fix,k} - 1 \quad \textrm{for all $k \in [K_{prior,fix}]$}
	\end{align*}
	By setting $\discp = \xi + 1$, $\concp = \lambda + 2$, $\rho_{fix,k} = \xi_{fix,k} + 1$, and
	$\sigma_{fix,k} = \lambda_{fix,k} - \xi_{fix,k} + 1$, we recover the hyperparameters of \eq{bp_params}
	in \ex{bp_bep}. Here, by contrast to \ex{bp_bep}, we found the conjugate prior and its hyperparameter
	settings given just the Bernoulli process likelihood.
	Henceforth, we use the parameterization of the beta process above.
\end{example}

\section{Further size-biased representations} \label{app:size}

\begin{example} \label{ex:bp_size}
	Let $\prm$ be a beta process, and let $\obsm_{n}$ be iid Bernoulli
	processes conditioned on $\prm$ for each $n$
	as in \ex{bep_auto_conj}. That is, we have
	$$
		\wratem(d\prw) = \massp \prw^{\xi} (1-\prw)^{\lambda - \xi} \; d\prw.
	$$
	And
	$$
		\denslike(\obsw | \prw_{k})
			= \prw_{k}^{\obsw} (1-\prw_{k})^{1-\obsw}
	$$
	with
	\begin{align*}
		\massp &> 0 \\
		\xi &\in (-2,-1] \\
		\lambda &> \xi - 1 \\
		\xi_{fix,k} &> -1 \textrm{ and } \lambda_{fix,k} > \xi_{fix,k} - 1 \quad \textrm{for all $k \in [K_{prior,fix}]$}
	\end{align*}
	by \ex{bep_auto_conj}.
	
	We can pick out the following components of $\denslike$:
	\begin{align*}
		\basem(\obsw) &= 1 \\
		\suffstat(\obsw) &= \obsw \\
		\natpar(\prw) &= \log\left( \frac{ \prw }{ 1-\prw } \right) \\
		\logpart(\prw) &= -\log(1-\prw).
	\end{align*}
	Thus, by \cor{size},
	 \begin{align*}
		\prm &= \sum_{m=1}^{\infty}
				\sum_{\obsw=1}^{\infty}
				\sum_{j=1}^{\rho_{m, \obsw}}
					\prw_{m, \obsw, j} \delta_{\prl_{m, \obsw, j}} \\
		\prl_{m, \obsw, j}
			&\iid \dordloc \quad \textrm{iid across $m, \obsw, j$} \\
		\prw_{m, \obsw, j}
			&\indep f_{size, m, \obsw}(\prw) \; d\prw \\
			&\propto \prw^{\xi + \obsw} (1-\prw)^{\lambda + m - \xi - \obsw} \; d\prw \\
			&\propto \tb\left( \prw 
				\left\vert
					\xi + \obsw, \lambda - \xi + m - \obsw
				\right. \right) \; d\prw \\
			& \textrm{iid across $j$ and independently across $m, \obsw$} \\
		M_{m, \obsw}
			&:= \massp
				\cdot \frac{
					\Gamma(\xi + \obsw + 1)
					\Gamma(\lambda - \xi + m - \obsw + 1)
				}{
					\Gamma(\lambda + m + 2)
				} \\
		\rho_{m, \obsw}
			&\indep \pois\left( M_{m, \obsw} \right) \\
			& \textrm{across $m, \obsw$}
	\end{align*}
	
	\citet{broderick:2012:beta} and  \citet{paisley:2012:stick} have previously
	noted that this size-biased representation of the
	beta process arises from the Poisson point process.
\end{example}

\section{Further marginals} \label{app:marg}

\begin{example} \label{ex:bp_bep_marg}
	Let $\prm$ be a beta process, and let $\obsm_{n}$ be iid Bernoulli
	processes conditioned on $\prm$ for each $n$
	as in \exs{bep_auto_conj} and \exss{bp_size}.
	
	We calculate the main components of \cor{marg} for this
	pair of
	processes. In particular, we have
	\begin{align*}
		\mbp(\obsw_{n} = 1) 
			&= \basem(k) \exp\left\{ -B(\xi + \sum_{m=1}^{n-1} \obsw_{m}, \lambda + n-1) + B(\xi + \sum_{m=1}^{n-1} \obsw_{m} + 1, \lambda + n) \right\} \\
			&= \frac{
					\Gamma(\lambda + n-1 + 2)
				}{
					\Gamma(\xi + \sum_{m=1}^{n-1} \obsw_{m} + 1)
					\Gamma(\lambda + n-1 - \xi - \sum_{m=1}^{n-1} \obsw_{m} + 1)
				} \\
			& \quad {} \cdot
				\frac{
					\Gamma( \xi + \sum_{m=1}^{n-1} \obsw_{m} + 1 + 1 )
					\Gamma( \lambda + n - \xi - \sum_{m=1}^{n-1} \obsw_{m} - 1 + 1 )
				}{
					\Gamma( \lambda + n + 2 )
				} \\
			&= \frac{
					\xi + \sum_{m=1}^{n-1} \obsw_{m} + 1
				}{
					\lambda + n + 1
				}
	\end{align*}
	And
	\begin{align*}
		M_{n, 1}
			&:= \massp \cdot \basem(0)^{n-1} \basem(1)
				\cdot \exp\left\{ B(\xi + (n-1) \suffstat(0) + \suffstat(1), \lambda + n) \right\} \\
			&= \massp
				\cdot \frac{
					\Gamma(\xi + 1 + 1)
					\Gamma(\lambda + n - \xi - 1 + 1)
				}{
					\Gamma(\lambda + n + 2)
				}
	\end{align*}

	Thus, the marginal distribution of $\obsm_{1:N}$ is the same as that provided by the 
	following construction.
	
	For each $n = 1, 2, \ldots, N$,
	\begin{enumerate}
		\item At any location $\prl$ for which there is some atom in $\obsm_{1}, \ldots, \obsm_{n-1}$,
		let $\obsw_{m}$ be the weight of $\obsm_{m}$ at $\prl$ for $m \in [n-1]$. Then
		we have that
		$\obsm_{n} | \obsm_{1}, \ldots, \obsm_{n-1}$ has weight $\obsw_{n}$ at $\prl$, where
		\begin{align*}
			\mbp(d\obsw_{n}) 
				&= \bern\left( \obsw_{n} 
					\left\vert
						\frac{
							\xi + \sum_{m=1}^{n-1} \obsw_{m} + 1
						}{
							\lambda + n + 1
						}
					\right.
				\right)
		\end{align*}
		\item $\obsm_{n}$ has $\rho_{n,1}$ atoms at locations $\{\prl_{n,1,j}\}$
			with $j \in [\rho_{n,1}]$ where there have not yet been atoms in any of $\obsm_{1},\ldots,\obsm_{n-1}$. Moreover,
			\begin{align*}
				M_{n, 1} &:= \massp
						\cdot \frac{
							\Gamma(\xi + 1 + 1)
							\Gamma(\lambda + n - \xi - 1 + 1)
						}{
							\Gamma(\lambda + n + 2)
						} \\
					& \textrm{ across $n$} \\
				\rho_{n,1} &\indep \pois\left( M_{n, 1} \right)
					\textrm{ across $n, \obsw$} \\
				\prl_{n,1,j} &\iid \dordloc(d\prl)
					\textrm{ across $n, j$}
			\end{align*}
	\end{enumerate}
	
	Here, we have recovered the three-parameter extension of the Indian buffet process
	\citep{teh:2009:indian, broderick:2013:cluster}.
\end{example}


\newpage
\bibliography{draft}
\bibliographystyle{plainnat}

\end{document}